\documentclass[preprint]{imsart}

\usepackage{mathtools}
\usepackage{booktabs}

\RequirePackage[OT1]{fontenc}
\RequirePackage{amsthm,amsmath}
\RequirePackage[round]{natbib}
\RequirePackage[colorlinks,citecolor=blue,urlcolor=blue]{hyperref}

\arxiv{arXiv:0000.0000}

\startlocaldefs
\allowdisplaybreaks

\newcommand{\bsc}{\boldsymbol{c}}

\newcommand{\bsx}{\boldsymbol{x}}
\newcommand{\bsy}{\boldsymbol{y}}
\newcommand{\bsz}{\boldsymbol{z}}

\newcommand{\real}{\mathbb{R}}

\newcommand{\tran}{\mathsf{T}} 

\newcommand{\rd}{\mathrm{\, d}}

\newcommand{\one}{{\mathbf{1}}}

\newcommand{\var}{{\mathrm{Var}}}

\newcommand{\vol}{{\mathbf{vol}}}

\newcommand{\wt}{\widetilde}

\renewcommand{\ge}{\geqslant}
\renewcommand{\le}{\leqslant}

\newcommand{\dexp}{\text{Exp}}

\newcommand{\dnorm}{\mathcal{N}}

\newcommand{\dustd}{\mathbf{U}} 

\newcommand{\e}{{\mathbb{E}}} 

\newcommand{\bby}{\mathbb{Y}}

\numberwithin{equation}{section}
\theoremstyle{plain}
\newtheorem{theorem}{Theorem} 

\allowdisplaybreaks[1]
\usepackage{graphicx,dsfont,amsmath,amssymb,url,xspace, amsthm,
  color,caption, subcaption}

\newcommand{\pluses}{+,  \cdots,  +}
\newcommand{\minuses}{-,  \cdots,  -}
\newcommand{\plusess}{+, +, +, \cdots, +, +, +}

\newcommand{\Kv}{K_{v, h, \bsx'}}

\newcommand{\bsr}{\boldsymbol{r}}

\newcommand{\beps}{\boldsymbol{\epsilon}}
\newcommand{\bseta}{\boldsymbol{\eta}}
\newcommand{\bsdel}{\boldsymbol{\delta}}

\newcommand{\Kvhepsc}{K_{v_{\boldsymbol{\epsilon}}, h_{\boldsymbol{\eta}}, \bsx_c}}

\newcommand{\Hv}{\mathcal{H}_{v, h, \bsx'}}
\newcommand{\hzx}{h(\left<\bsz, \bsx \right>)}

\newcommand{\heps}{h_{\boldsymbol{\eta}}}
\newtheorem{lemma}{Lemma}
\newtheorem{proposition}{Proposition}

\newtheorem{model}{Reference distribution}

\newtheorem*{remark}{Remark}

\newcommand{\sd}{{\mathbb{S}^d}}
\newcommand{\sdmo}{{\mathbb{S}^{d-1}}}
\newcommand{\sdmt}{{\mathbb{S}^{d-2}}}
\newcommand{\minm}{\underbar{m}}
\newcommand{\minr}{\underbar{r}}
\newcommand{\simiid}{\stackrel{\mathrm{iid}}\sim}
\newcommand{\phz}{\phantom{0}}

\renewcommand{\le}{\leqslant}
\renewcommand{\ge}{\geqslant}

\endlocaldefs

\begin{document}

\begin{frontmatter}
\title{Permutation $p$-value approximation via generalized Stolarsky invariance}
\runauthor{He, Basu, Zhao \& Owen}
\runtitle{Stolarsky invariance for permutations}
\begin{aug}
\author{\fnms{Hera Y.} \snm{He,}\ead[label=e1]{(yhe1, kinjal, qyzhao, owen)@stanford.edu}}
\author{\fnms{Kinjal} \snm{Basu,}\ead[label=e2]{kinjal@stanford.edu}}
\author{\fnms{Qingyuan} \snm{Zhao,}\ead[label=e3]{qyzhao@stanford.edu}}
\author{\fnms{Art B.} \snm{Owen}\ead[label=e4]{owen@stanford.edu}}
\affiliation{Stanford University}
\address{Department of Statistics\\
Sequoia Hall\\
Stanford University\\
Stanford, CA 94305\\
\printead{e1}
}
\end{aug}
\date{}

\begin{abstract}
It is common for genomic data analysis to use $p$-values from 
a large number of permutation tests. The multiplicity of tests may require
very tiny $p$-values in order to reject any null hypotheses and the
common practice of using randomly sampled permutations then
becomes very expensive.  We propose an inexpensive approximation
to $p$-values for two sample linear test statistics, derived from
Stolarsky's invariance principle. The method creates a geometrically
derived set of approximate $p$-values for each hypothesis. 
The average of that set is used as
a point estimate $\hat p$ and our generalization of 
the invariance principle allows us to compute
the variance of the $p$-values in that set.  We find that in cases
where the point estimate is small the variance is a modest multiple
of the square of the point estimate, yielding a relative error property
similar to that of saddlepoint approximations.
On a Parkinson's disease data set, the new approximation 
is faster and more accurate than the saddlepoint approximation.
We also obtain a simple probabilistic explanation of Stolarsky's
invariance principle.
\end{abstract}
\end{frontmatter}
\section{Introduction}

Permutation methods are commonly used to obtain $p$-values
in genomic applications.   They make only modest assumptions
and they have a direct intuitive interpretation that appeals to biologists
in collaborations.   In even modestly large data
sets, the exact permutation $p$-value becomes
too expensive to compute. Then Monte Carlo sampling of random permutations
becomes a standard approach.  
Genomic applications commonly require thousands or more
of hypotheses to be tested, and then multiplicity adjustment requires that 
some small  $p$-values be obtained if any null hypotheses
are to be rejected.
When $p$-values below $\epsilon$ are required to reject $H_0$, then 
\cite{knij:wess:rein:shmu:2009} 
recommend doing at least $10/\epsilon$ random permutations. 
As a result, even Monte Carlo sampling for permutation tests
can be prohibitively expensive, and hence it pays to search
for fast approximations to the permutation $p$-value.

In this paper we develop rapidly computable  approximations
to some permutation $p$-values. The $p$-values we
consider are for a difference in group means. The approximations are based
on ideas from spherical geometry and discrepancy,
related to the Stolarsky invariance principle
\citep{stol:1973}.  As described below, the resulting approximations
prove to be very accurate for the tiny $p$-values where permutation
methods are most difficult to use.

We begin with some background on the genomic motivation
of our work. Then we transition to  spherical geometry.

\subsection*{Genomic context}
The specific problem that motivated us is testing for sets of genes associated
with Parkinson's disease \citep{lars:owen:2015}.  
More details about this work are given in the the first author's 
dissertation \citep{he:2016}.

In these data sets, there are $m_0$ subjects without Parkinson's disease
and $m_1$ subjects with it.  One can test whether Parkinson's disease is associated
with an individual gene by doing a $t$-test comparing gene expression
levels in tissue samples from the two groups of subjects.  Biological interest is often
summarized more by gene sets rather than individual genes.
Gene sets have two advantages:  small but consistent associations of many
genes with the test condition can raise power, and, the gene sets themselves
often connect better to biological understanding than do individual genes.
The null hypothesis is that the subject condition does not
affect expression levels of any gene in the gene set.

There is a very large literature on testing for significant associations
between a condition and the genes in a gene set.  
See \cite{acke:stri:2009} for an overview of the main concepts 
and methods in that literature. 
They did an extensive comparison of $261$ different
gene set testing methods and identified two families of winning methods.  
Let $t_g$ be the ordinary two sample $t$ statistic for
comparing the average expression level of gene $g$ between two
conditions and let $G$ be a set of genes of interest.
They found that linear and quadratic test statistics
$L_G=\sum_{g\in G}t_g$ and $Q_G=\sum_{g\in G}t_g^2$ 
had the best power, along with some simple approximations to those
two statistics.  These methods performed better
than some subtantially more complicated proposals.
The statistic $L_G$ was proposed by \cite{jian:gent:2007}.
The statistic $L_G$ and similar ones did best when expression 
differences between the two conditions tend to have the
same sign, for each $g\in G$.  If large but oppositely signed
treatment effects occur, then $Q_G$ and approximations to it do best.

The genes in a gene set are ordinarily correlated with each other,
even if they are independent of the treatment condition. The correlations
makes it difficult to find the null distributions of $L_G$ and $Q_G$,
even with parametric model assumptions.
In a permutation analysis, like \cite{acke:stri:2009} use,
we consider all $N={n\choose m_1}$ 
different ways to select a subset $\pi$ containing $m_1$ 
of the $n=m_0+m_1$ subjects. Let $Q_G^\pi$
be the test statistic recomputed as if those $m_1$ subjects had
been the affected group. Then the permutation $p$-value for $Q_G$ is
$$
p = p_G =\frac1N\sum_{\pi} \one_{Q_G^\pi \ge Q_G}.
$$
Note that the smallest possible value for $p$ is $1/N$.

When $N$ is too large for a permutation test
to be computationally feasible, a standard practice
is to estimate $p$ via randomly sampled permutations of the treatment label
as proposed by  \cite{barn:1963}.  
For $\ell=1,\dots,M-1$ we let $\pi(\ell)$ be the affected group after a
randomization of the treatment labels.  We let $\pi(0)$ be the original allocation.
Then the Monte Carlo estimate
$$
\hat p = \frac1M\sum_{\ell=0}^{M-1}\one_{Q_G^{\pi(\ell)} \ge Q_G}
$$
is used as an estimate of $p$, for the quadratic statistic.
In this Monte Carlo, the true permutation $p$-value $p$ is the unknown
parameter and $\hat p$ is the sample estimate of $p$.
Note that $\hat p\ge 1/M$ because we have included the original
allocation in the numerator.  Failure to include the original allocation $\pi(0)$
can lead to $\hat p =0$ which is very undesirable.
We call $1/M$ the granularity limit. 
When $p$ is quite small, an enormous number $M$ of simulations may be required
to get an accurate estimate of it.
For instance, in genome wide association studies (GWAS)
the customary threshold for significance is $\epsilon = 5\times 10^{-8}$,
making permutation methods prohibitively expensive, 
or even infeasible.
For a recent discussion of $p$-value thresholds in GWAS, see \cite{fadi:etal:2016}. 


In this paper, we work with one of
\nocite{acke:stri:2009}  Ackermann and Strimmer's (2009)
approximations to $L_G$.
Let $X_i = 1$ if subject $i$ is in condition $1$ and $X_i=0$ for
condition $0$. Let $Y_{ig}$ be the expression level of gene $g$
for subject $i$.  Let $\hat\rho_g$ be the sample correlation
between $X_i$ and $Y_{ig}$. 
Then $t_g = \sqrt{n-2}\hat\rho_g/\sqrt{1-\hat\rho_g^2}$
and a first order Taylor approximation gives
$ t_g \doteq \sqrt{n-2}(\hat\rho_g +\hat\rho^3_g/2)$.
When many small correlations $\hat\rho_g$ contribute to the
signal, then summing $\hat\rho_g$ gives a test statistic 
that is almost equivalent to summing $t_g$.  
\cite{acke:stri:2009} found that
\begin{align}\label{eq:sumofcorr}
\sum_{g\in G} \frac1n\sum_{i=1}^n \frac{X_i-\bar X}{s_X}\frac{Y_{gi}-\bar Y_g}{s_{g}}
\end{align}
was in the same winning set of methods as $L_G$, where $s_X$ and $s_{g}$
are standard deviations of $X_i$ and $Y_{gi}$ respectively.
They also considered using pooled variance estimates
in place of $s_g$ but found no advantage to doing so, perhaps because
$\min(m_0,m_1)$ was at least $10$ in their simulations.
Letting $Y_i = Y_{Gi} \equiv \sum_{g\in G}Y_{gi}/s_g$,
we may rewrite~\eqref{eq:sumofcorr} as
\begin{align}\label{eq:sumofcorr2}
\sum_{i=1}^n \frac{X_i-\bar X}{\Vert X-1_n\bar X\Vert}
\frac{Y_{i} -\bar Y}{\Vert Y-1_n\bar Y\Vert}
\end{align}
multiplied by a constant that only depends on $n$ and hence
does not affect $p$.
Equation~\eqref{eq:sumofcorr2} describes a test statistic that is a
plain Euclidean inner product
of two unit vectors in $\real^n$.
Here $\bsx_0$ has $i$'th component 
$(X_i-\bar X)/\Vert X-1_n\bar X\Vert$
and $\bsy_0$ is similar.

There are $N={n\choose m_1}$ distinct vectors found by
permuting the entries in $\bsx$.
We label them $\bsx_0,\bsx_1,\dots,\bsx_{N-1}$ with
$\bsx_0$ being the original one.
Letting $\hat\rho =\bsx_0^\tran\bsy_0$ we find that
one and two-sided $p$-values for a linear statistic  are
$$
\frac1N\sum_{k=0}^{N-1}\one_{\bsx_k^\tran\bsy_0\ge\hat\rho}\quad\text{and}\quad
\frac1N\sum_{k=0}^{N-1}\one_{|\bsx_k^\tran\bsy_0|\ge|\hat\rho|}
$$
repectively. We prefer two-sided test statistics, but we will study
one-sided ones first and then translate our results to two-sided ones.

\subsection*{Spherical geometry}
We are now ready to make a geometric interpretation.
Let $\sd=\{\bsz\in\real^{d+1}\mid \bsz^\tran\bsz=1\}$
be the $d$-dimensional unit sphere. 
Our data $\bsx_0,\bsy_0$ are in a subset of $\mathbb{S}^{n-1}$
orthogonal to $1_n$.  That subset is isomorphic to
$\mathbb{S}^{n-2}$ and so we work mostly with $d=n-2$.

Given a point $\bsy\in\sd$, the points $\bsz$ that are closest
to $\bsy$ comprise a spherical cap. The spherical cap of center $\bsy$
and height $t\in[-1,1]$ is
$C(\bsy;t) = \{\bsz\in\sd\mid \langle \bsy,\bsz\rangle\ge t\}$.
By symmetry, $\bsz\in C(\bsy;t)$ if and only if $\bsy\in C(\bsz;t)$.
The one-sided linear $p$-value is the fraction of $\bsx_k$ for $0\le k<N$
that belong to $C(\bsy_0;\hat\rho)$.  
A natural, but crude approximation to $p$ is then
$$\hat p_1(\hat\rho),\quad\text{where} \quad
\hat p_1(t)
\equiv
\frac{\vol( C(\bsy_0;t))}{\vol(\sd)}.$$

Stolarsky's invariance principal gives a remarkable description of the
accuracy of this approximation $\hat p_1$.
The squared $L_2$ spherical cap discrepancy of points
$\bsx_0,\bsx_1,\dots,\bsx_{N-1}\in\sd$  is
$$
L_2(\bsx_0,\dots,\bsx_{N-1})^2
=\int_{-1}^1\int_{\sd} |\hat p_1(t)-p(\bsz,t)|^2\rd\sigma_d(\bsz)\rd t
$$
where $p(\bsz,t)=(1/N)\sum_{k=0}^{N-1}\one_{\bsx_k\in C(\bsz,t)}$ and
$\sigma_d$ is the uniform (Haar) measure on $\sd$.
\cite{stol:1973} shows that
\begin{equation}\label{eq:stolarsky}
\begin{split}
\frac{d\omega_d}{\omega_{d+1}} \times L_2(\cdot)^2
&= \int_{\sd}\int_{\sd}\|\bsx-\bsy\|\rd \sigma_d(\bsx)\rd \sigma_d(\bsy)
-\frac{1}{N^2}\sum\limits_{k, l=0}^{N-1}\|\bsx_k - \bsx_l\|
\end{split}
\end{equation}
where $\omega_d$ is the (surface) volume of $\sd$.
Equation~\eqref{eq:stolarsky} relates the mean squared error of $\hat p_1$
to the mean absolute Euclidean distance among the $N$ points.
In our applications, the $N$ points will be the distinct permuted values of 
$\bsx_0$, but~\eqref{eq:stolarsky}
holds for an arbitrary set of $N$ points $\bsx_k$.

The left side of~\eqref{eq:stolarsky} is, up to normalization, a mean squared
discrepancy over spherical caps.
This average of $(\hat p_1 -p)^2$ includes $p$-values of all sizes between
$0$ and $1$. It is not then a very good accuracy measure when $\hat p_1(\hat\rho)$
turns out to be very small, such as $10^{-6}$.
It would be more useful to get such a mean squared error taken over
caps of exactly the size $\hat p_1(\hat\rho)$, and no others.

\cite{brau:dick:2013} consider quasi-Monte Carlo (QMC) sampling in the sphere.
They generalize Stolarsky's discrepancy formula to include a weighting
function on the height $t$.  By specializing their formula, we get
an expression for the mean of $(\hat p_1-p)^2$ over spherical caps of any fixed size.

Discrepancy theory plays a prominent role in QMC
\citep{nied:1992}, which is about approximating an integral by a sample average.
The present setting is  a reversal of QMC: the discrete average $p$
over permutations is the exact value we seek, and the integral over a continuum
is the approximation $\hat p$. A second difference is that the QMC
literature focusses on choosing $N$ points to minimize a criterion
such as~\eqref{eq:stolarsky}, whereas
here the $N$ points are determined by the problem.


As we will show below, the estimate $\hat p_1$ is the average of $p$ over
all spherical caps $C(\bsy;\hat\rho)$ 
under a uniform distribution, i.e.,  $\bsy\sim\dustd(\sd)$.
Those caps have the same volume as $C(\bsy_0;\hat\rho)$.

In addition to specializing from caps $C(\bsy;t)$ with $t=\hat\rho$
we can also specialize to caps whose centers $\bsy$ 
more closely resemble $\bsy_0$.
Suppose that
$$
\bsy_0 \in \bby \subset \sd.
$$
Then we know that $p(\bsy_0,\hat\rho) \le \sup_{\bsy\in \bby}p(\bsy,\hat\rho)$,
which is a conservative permutation $p$-value.
We are generally unable to compute this quantity but in some instances
we can form a reference distribution $\bsy\sim\dustd(\bby)$ and
compute both $\e(p(\bsy,\hat\rho)\mid\bsy\in\bby)$
and $\var(p(\bsy,\hat\rho)\mid\bsy\in\bby)$, the mean and
variance of $p(\bsy,\hat\rho)$ under this distribution.

The simplest reference distribution we use has
$\bby_1=\sd$.  
We have found that the set
$\bby_2 = \{\bsy\in\sd \mid \bsy^\tran\bsx_0 = \bsy_0^\tran\bsx_0\}$
and some generalizations yield especially useful reference distributions.
Generalizations of the Stolarsky formula allow us
to compute $\var(p(\bsy,\hat\rho)\mid\bsy\in\bby_2)$.
 In some of our numerical results from Section~\ref{sec:experimental},
we find that $\var(p(\bsy,\hat\rho)\mid\bsy\in \bby_2)$ is so small
that the true permutation $p$-value  $p(\bsy_0,\hat\rho)$ must
be of the same order of magnitude as the estimate
$\hat p_2\equiv \e(p(\bsy,\hat\rho)\mid\bsy\in\bby_2)$
that we study at length.

We obtain $\hat p_2$  and the mean square discrepancy
over its reference distribution by further extending Brauchart and Dick's
generalization of Stolarsky's invariance. 
More generally, we can replace the constraint
$\langle\bsy,\bsx_0\rangle = \langle\bsy_0,\bsx_0\rangle$
by 
$\langle\bsy,\bsx_c\rangle = \langle\bsy_0,\bsx_c\rangle$
for any individual $c\in \{0,1,\dots,N-1\}$.
Our estimate $\hat p_3$ takes
$\bsx_c$ to be whichever permuted
point $\bsx_k$ happens to be closest to $\bsy_0$.

Smaller sets $\bby$ could be even better than $\bby_2$. 
In the extreme, if we could work with $\bby$ that satisfies
$\langle\bsy,\bsx_k\rangle=\langle\bsy_0,\bsx_k\rangle$,
for  all  $0\le k<N$ then all $\bsy\in\bby$ would have 
$p(\bsy,\hat\rho)=p(\bsy_0,\hat\rho)=p$ and the mean over $\bsy\in\bby$
would have no error.  
Any smaller set is only useful if we can efficiently compute with it.

Although we found these results via invariance, we can also obtain
them via probabilistic arguments. As a consequence we have a probabilistic
derivation of Stolarsky's formula.
\cite{bilyk2016stolarsky} have independently found this connection.
Some of our results are for arbitrary $\bsx$, but our
best computational formulas are for the case where the variable $\bsx$
is binary, as it is for the Parkinson's disease data sets.

\subsection*{Outline}
The rest of the paper is organized as follows. Section~\ref{sec:background}
presents some context on permutation tests and gives some results from spherical geometry.
In Section~\ref{sec:stolarskys} we use Stolarsky's invariance
principle as generalized by \cite{brau:dick:2013} to obtain the mean squared
error between the true $p$-value and its continuous approximation $\hat p_1$, 
averaging over all spherical caps of volume $\hat p_1$.
This section also has a probabilistic derivation of that mean squared error.
In Section~\ref{sec:finerapprox} we
describe some finer approximations $\tilde p$ for the $p$-value.
These use the set $\bby_2$ to condition on not just the volume of the spherical 
cap but also on its distance from the original data point $\bsx_0$, 
or from some
other point, such as the closest permutation of $\bsx_0$ to $\bsy_0$.
By always including the original point we ensure that $\tilde p\ge 1/N$. That is a desirable
property because the true permutation $p$-value cannot be smaller than $1/N$.
In Section~\ref{sec:genstolarsky}
we modify the proof in \cite{brau:dick:2013}, to further generalize
their invariance results to include the mean squared error of the finer approximations.
Section~\ref{sec:two-sided} extends our estimates to two-sided testing.
Section~\ref{sec:experimental} illustrates our $p$-value approximations numerically.
We see that an RMS error in the estimate $\hat p_2$ is 
of the same order of magnitude as $\hat p_2$ itself. That is,
$\hat p_2$ has a relative error property like saddlepoint estimates do.
Section~\ref{sec:saddle} makes a numerical comparison to saddlepoint methods
in simulated data.  The saddlepoint estimates come out more
accurate than $\hat p_2$ but are biased low in the simulated examples.
Section~\ref{sec:data} compares the accuracy of our approximations 
to each other and to the saddlepoint approximation for
6180 gene sets and some Parkinson's disease data sets. In the data
examples, the new approximations come out closer to some gold
standard estimates (based on large Monte Carlo samples) than the saddlepoint
estimates do, which once again are biased low.
From Table 6.3 of \cite{he:2016}, the saddlepoint computations take
roughly $30$ times longer than $\hat p_2$ does.
Section~\ref{sec:discussion} draws some conclusions and
discusses the challenges in getting a computationally feasible
$p$-value that accounts for both sampling uncertainty of $(\bsx,\bsy)$
and the uncertainty in $\hat p$ as an estimate of $p$.
Most of the proofs are in the Appendix, Section~\ref{sec:appendix}.

\subsection*{Software}

The proposed approximations are implemented in the R package pipeGS on CRAN. Given a binary input label and a gene expression measurement matrix, it computes 
our three p value approximations
for the linear gene set statistics. 
Those statistics, $\hat p_1$, $\hat p_2$, $\hat p_3$, are mentioned
above and then presented in more detail in Section~\ref{sec:finerapprox}.
We provide an implementation of the saddlepoint approximation in the package as well.

\section{Background and notation}\label{sec:background}

The raw data contain points $(X_i,Y_i)$ for $i=1,\dots,n$, where
$Y_i$ may be a  composite quantity derived from 
all $Y_{gi}$ for $g$ belonging to a gene set $G$, such as
$Y_i=Y_{Gi}$ just before~\eqref{eq:sumofcorr2}.
We center and scale vectors 
$(X_1,X_2,\dots,X_n)$ and $(Y_1,Y_2,\dots,Y_n)$
yielding $\bsx_0,\bsy_0\in\sd$ for $d=n-1$.
Both points belong to $\{\bsz\in\mathbb{S}^{n-1}\mid \bsz^\tran 1_n=0\}$. 
We can use an orthogonal matrix to rotate the points of this set onto 
$\mathbb{S}^{n-2}\times\{0\}$.  As a result, we may simply work with 
$\bsx_0,\bsy_0\in\sd$ where $d=n-2$. 

The sample correlation of these variables is
$\hat\rho = \bsx_0^\tran\bsy_0 = \langle \bsx_0,\bsy_0\rangle$.
We use $\langle \bsx_0,\bsy_0\rangle$ when we find that geometrical
thinking is appropriate and to conform with \cite{brau:dick:2013}. 
We use $\bsx_0^\tran \bsy_0$ to emphasize computational 
or algebraic connotations.

Here we develop approximations
to the one-sided $p$-value as that simplifies notation.
Section~\ref{sec:two-sided}
shows how to obtain the corresponding two-sided $p$-values.
 We assume that $\hat\rho>0$ for otherwise $\hat p_1$
is going to be too large to be interesting.  For instance
with $m_0=m_1$, $\hat\rho \le 0$ implies that $\hat p_1 \ge 1/2$.

Our proposals are computationally most attractive in the case where
$X_i$ takes on just two values, such as $0$ and $1$.
Then $\hat\rho$ is a two-sample test statistic for a difference in means.
When there are $m_0$ observations with $X_i=0$ and $m_1$ with $X_i=1$
then $\bsx_0$ contains $m_0$ components equal to $-\sqrt{m_1/(nm_0)}$
and $m_1$ components equal to $+\sqrt{m_0/(nm_1)}$.
Computational costs are often sensitive to the smaller sample size, 
$\minm \equiv \min(m_0,m_1)$.

For this two-sample case there are only $N={m_0+m_1\choose m_0}$ distinct
permutations of $\bsx_0$.
We have called these $\bsx_0,\bsx_1,\dots,\bsx_{N-1}$  and 
the true $p$ value is
$
p = (1/N)\sum_{k=0}^{N-1} \one(\bsx_k^\tran\bsy_0 \ge \hat{\rho})
= (1/N)\sum_{k=0}^{N-1} \one(\bsx_k\in C(\bsy_0;\hat\rho))$.

Now suppose that there are exactly $r$ indices for which $\bsx_k$ is positive
and $\bsx_\ell$ is negative.  There are then $r$ indices with the reverse pattern too.
We say that $\bsx_k$ and $\bsx_\ell$ are at `swap distance $r$'
because $r$ zeros from $\bsx_k$ swapped positions with ones to yield $\bsx_\ell$.
In that case we easily find that
\begin{align}\label{eq:ur}
  u(r) \equiv 
\langle \bsx_k, \bsx_\ell \rangle
= 1 - r\Bigl(\frac1{m_0}+\frac1{m_1}\Bigr).
\end{align}

We need some geometric properties of the unit sphere and spherical caps.
The surface volume of $\sd$ is $\omega_d = 2\pi^{(d+1)/2}/\Gamma((d+1)/2)$.
We use $\sigma_d$ for the volume element in $\sd$ normalized so that
$\sigma_d(\sd)=1$.
The spherical cap $C(\bsy;t)
=\{\bsz\in\sd\mid \bsz^\tran\bsy\ge t\}$ has volume
\begin{align*}
\sigma_d(C(\bsy; t)) =
\begin{cases}
\frac{1}{2}I_{1-t^2}\left(\frac{d}{2}, \frac{1}{2}\right),  &0 \le t\le 1\\
1 - \frac{1}{2}I_{1-t^2}\left(\frac{d}{2}, \frac{1}{2}\right),  & -1 \le t < 0
\end{cases}
\end{align*}
where $I_t(a, b)$ is the incomplete beta function
\begin{align*}
I_t(a, b) =\frac{1}{B(a, b)}\int_0^t x^{a-1}(1-x)^{b-1} \rd x
\end{align*}
with $B(a, b) = \int_0^1 x^{a- 1}(1-x)^{b-1} \rd x$.
Obviously, this volume is $0$ if $t<-1$ and it is $1$ if $t>1$.
This volume is independent of $\bsy$ so we may write
$\sigma_d( C(\cdot\,,t))$ for the volume.

Our first approximation of the $p$-value is
  $\hat{p}_1(\hat{\rho}) = \sigma_d(C(\bsy;\hat{\rho}))$.
We remarked earlier that this approximation
equates a discrete fraction to a volume ratio.
We show in Proposition~\ref{prop:unbiased} that
$\hat{p}_1=\e( p\mid \left<\bsx_0,\bsy \right> = \hat{\rho})$ for $\bsy\sim\dustd(\sd)$
as $\bsy_0$ would if the original $Y_i$ were IID Gaussian.
In Theorem \ref{thm:phat1}, we find $\var(\hat p_1)$ under this assumption.

We frequently need to project $\bsy\in\sd$ onto a point $\bsx\in \sd$.
In this representation $\bsy = t\bsx +\sqrt{1-t^2}\bsy^*$
where $t=\bsy^\tran\bsx\in[-1,1]$ and
$\bsy^* \in \{\bsz\in\sd\mid\bsz^\tran\bsx=0\}$ which is isomorphic to $\sdmo$.
The coordinates $t$ and $\bsy^*$ are unique.
From equation (A.1) in \cite{brau:dick:2013} we get
\begin{align}\label{eq:decomp}
\rd\sigma_d(\bsy) =\frac{\omega_{d-1}}{\omega_d}(1-t^2)^{d/2-1}\rd t
\rd\sigma_{d-1}(\bsy^*).
\end{align}
In their case $\bsx$ was $(0,0,\dots,1)$.

The intersection of two spherical caps of common height $t$ is
$$C_2(\bsx,\bsy;t) \equiv C(\bsx;t)\cap C(\bsy;t). $$
We will need the volume of this intersection.
\cite{lee:kim:2014} give a general solution for spherical cap intersections
without requiring equal heights. They enumerate $25$ cases, but
our case does not correspond to any single case of theirs and so
we obtain the formula we need directly, below.
We suspect it must be known already, but
we were unable to find it in the literature.

\begin{lemma}\label{lem:volintersection}
Let $\bsx,\bsy\in\sd$ and $-1\le t\le 1$ and put $u=\bsx^\tran\bsy$.
Let $V_2(u;t,d) = \sigma_d(C_2(\bsx, \bsy;t))$.
If $u=1$, then $V_2(u;t,d) = \sigma_d(C(\bsx;t))$.
If $-1<u<1$, then
\begin{align}\label{eq:volintersection}
V_2(u;t,d) =
\frac{\omega_{d-1}}{\omega_d}
\int_{t}^1 (1-s^2)^{\frac{d}2-1}
\sigma_{d-1}( C(\bsy^{*};\rho(s)))\rd s,
\end{align}
where $\rho(s) = (t-su)/\sqrt{(1-s^2)(1-u^2)}$.
Finally, for $u=-1$,
\begin{align}\label{eq:volintersection2}
V_2(u;t,d)=
\begin{cases} 0, & t\ge 0\\
\frac{\omega_{d-1}}{\omega_d}
\int_{-|t|}^{|t|}(1-s^2)^{\frac{d}2-1} \rd s,&\mathrm{else.}
\end{cases}
\end{align}

\end{lemma}

\begin{proof}
Let $\bsz\sim\dustd(\sd)$.
Then $V_2(u;t,d)=\sigma_d(C_2(\bsx,\bsy;t))=\Pr( \bsz\in C_2(\bsx,\bsy;t))$.
If $u=1$ then $\bsx=\bsy$ and so $C_2(\bsx,\bsy;t)=C(\bsx;t)$.
For $u<1$,
we project $\bsy$ and $\bsz$ onto $\bsx$, via
$\bsz  = s\bsx + \sqrt{1-s^2}\bsz^{*}$ and
$\bsy  = u\bsx + \sqrt{1-u^2}\bsy^{*}$.
Now
\begin{align*}
V_2(u;t,d)
&= \int_{\sd}
\one(\langle\bsx, \bsz \rangle \ge t)
\one(\langle\bsy, \bsz \rangle \ge t)
\rd \sigma(\bsz)\\
&=
\int_{-1}^1\one(s\ge t)\frac{\omega_{d-1}}{\omega_d}(1-s^2)^{\frac{d}{2}-1} \\
&\quad \times \int_{\sdmo}\!\!\one(su +
\sqrt{1-s^2}\sqrt{1-u^2}\left<\bsy^{*}, \bsz^{*} \right> \ge t)
\rd \sigma_{d-1}(\bsz^{*}) \rd s.
\end{align*}
If $u>-1$ then this reduces to~\eqref{eq:volintersection}. For $u=-1$ we get
$$
V_2(u;t,d) = \frac{\omega_{d-1}}{\omega_d}
\int_{-1}^1\one(s\ge t)\one(-s\ge t)
(1-s^2)^{\frac{d}{2}-1} \rd s.
$$
which reduces to~\eqref{eq:volintersection2}.
\end{proof}

When we give probabilistic arguments and interpretations we
do so for a random center $\bsy$ of a spherical cap.
That random center is taken from two reference distributions.
Those are distributions~\ref{mod:one} and~\ref{mod:two} below.
Reference distribution~\ref{mod:one} is illustrated in Figure~\ref{fig:mod1}.
Distribution~\ref{mod:two} is illustrated in Figure~\ref{fig:mod2}
of Section~\ref{sec:finerapprox} where we first use it.

\begin{model}\label{mod:one}
The vector $\bsy\sim\dustd(\bby_1)$ where $\bby_1=\sd$.
Expectation under this distribution is denoted $\e_1(\cdot)$.
\end{model}
\begin{model}\label{mod:two}
The vector
$\bsy\sim\dustd(\bby_2)$ where
$$\bby_2=\{\bsz\in \sd\mid\bsz^\tran\bsx_c=\tilde\rho\},$$
for some $-1 \le \tilde{\rho} \le 1$, and $c\in\{0,1,\dots,N-1\}$.
Then $\bsy=\tilde{\rho}\bsx_c + \sqrt{1 - \tilde{\rho}^2}\bsy^{*}$
for $\bsy^{*}$ uniformly distributed on a subset of $\sd$
isomorphic to $\sdmo$.
Expectation under this distribution is denoted $\e_2(\cdot)$.
\end{model}


Reference distribution~\ref{mod:one}
 holds true if the $Y_i$
are IID Gaussian random variables (with positive variance).
In that case, the estimate $\hat p_1$ is the same as we would get under
a $t$-test.  
Reference distribution~\ref{mod:two} is a significant narrowing
of reference distribution~\ref{mod:one} in the direction of
the ultimate reference distribution: a point mass on $\bsy=\bsy_0$.

\begin{figure}[t]
  \centering
  \includegraphics[scale = 0.35]{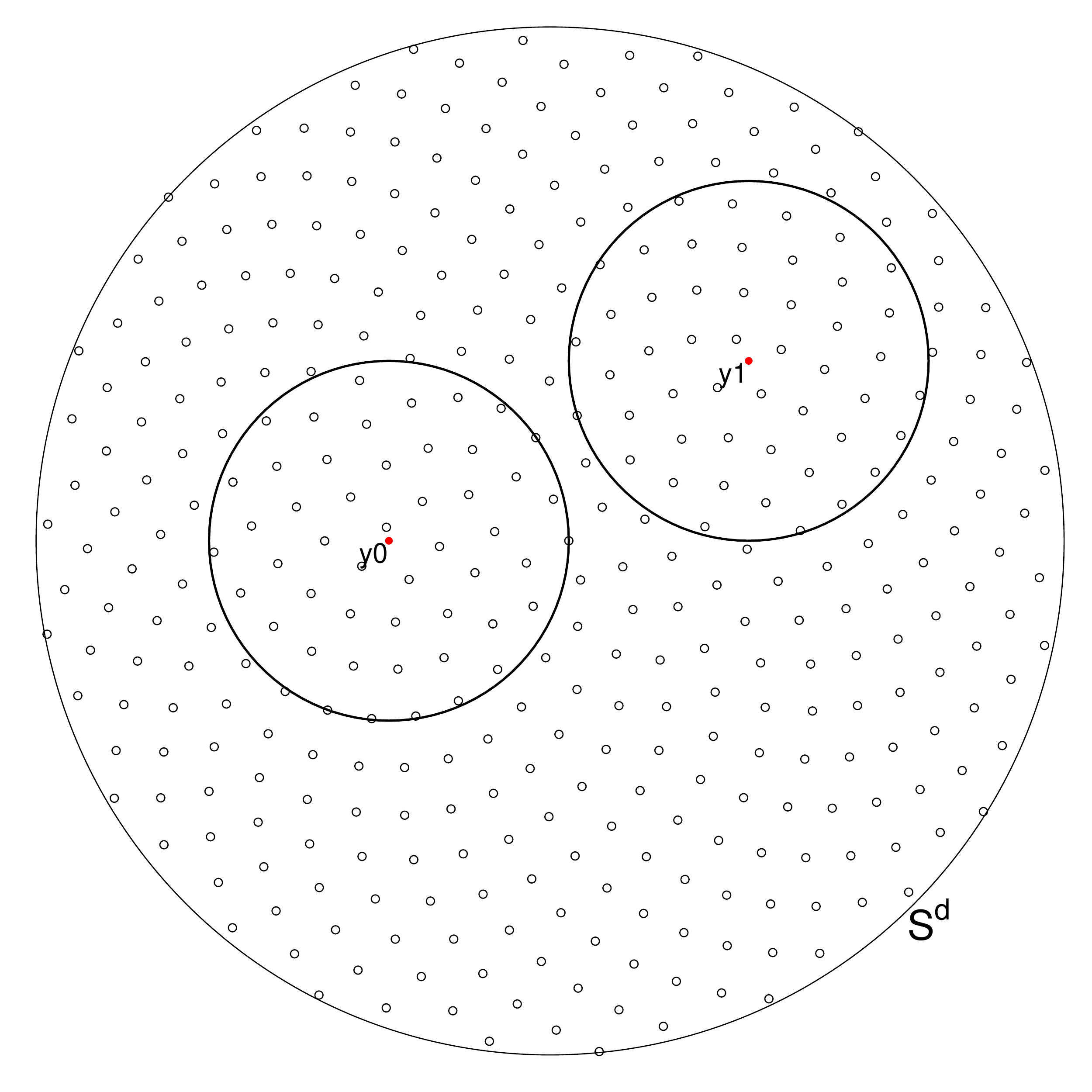}
  \caption{ \label{fig:mod1}
Illustration for reference distribution~\ref{mod:one}.
The point $\bsy$ is uniformly distributed over $\sd$. 
The small open circles represent permuted vectors $\bsx_k$. 
The point $\bsy_0$ is the observed value of $\bsy$.
The circle around it goes through $\bsx_0$ and represents a spherical cap
of height $\bsy_0^\tran\bsx_0$.
A second spherical cap of equal volume is centered at $\bsy=\bsy_1$.
We study moments of $p(\bsy,\hat\rho)$, the fraction of $\bsx_k$
in the cap centered at random~$\bsy$.
}
\end{figure}

\section{Approximation via spherical cap volume}\label{sec:stolarskys}

Here we study the approximate $p$-value
$\hat{p}_1(\hat{\rho}) = \sigma_d(C(\bsy;\hat{\rho}))$.
First we find the mean squared error of this approximation
over all spherical caps of the given volume via invariance.
Next we give a probabilistic interpretation which includes
the conditional unbiasedness result in Proposition~\ref{prop:unbiased} below.
Then we give two computational simplifications, first taking advantage
of the permutation structure of our points,
and then second for permutations
of a binary vector. We begin by restating the invariance principle.

\begin{theorem}\label{thm:stolarsky}
Let $ \bsx_0, \ldots, \bsx_{N-1}$ be any points in $\sd$. Then
\begin{align*}
\frac{1}{N^2}\sum\limits_{k, \ell=0}^{N-1}\|\bsx_k - \bsx_\ell\|  &+ \frac{1}{C_d}\int_{-1}^1\int_{S_d}
\biggl|\sigma_d(C(\bsz; t)) -
\frac{1}{N}\sum\limits_{k = 0}^{N-1}\one_{C(\bsz; t)}(\bsx_k)
\biggr|^2\rd \sigma_d(\bsz) \rd t \\
&= \int_{\sd}\int_{\sd}\|\bsx-\bsy\|\rd \sigma_d(\bsx)\rd \sigma_d(\bsy)
\end{align*}
where $C_d = {\omega_{d-1}}/(d\omega_{d})$.
\end{theorem}
\begin{proof}
\cite{stol:1973}.
\end{proof}

\cite{brau:dick:2013}
gave a simple proof of Theorem~\ref{thm:stolarsky}
using reproducing kernel Hilbert spaces. They also generalized
Theorem~\ref{thm:stolarsky} as follows.

\begin{theorem}\label{thm:weighted-stolarsky}
Let $\bsx_0, \ldots, \bsx_{N-1}$ be any points in $\sd$.
Let $v:[-1,1]\to(0,\infty)$ be any function with an antiderivative.
Then
  \begin{align}  \label{eq:brauthm}
  \begin{split}
    &\quad\int_{-1}^1v(t) \int_{\sd} \biggl| \sigma_d(C(\bsz;t)) -
      \frac{1}{N}\sum\limits_{k=0}^{N-1}\one_{C(\bsz;t)}(\bsx_k)\biggr|^2
    \rd \sigma_d(\bsz) \rd t \\
   & =
    \frac{1}{N^2}\sum\limits_{k, l = 0}^{N-1}K_{v}(\bsx_k, \bsx_\ell) - \int_{\sd} \int_{\sd}
    K_{v}(\bsx, \bsy) \rd \sigma_d(\bsx)\rd \sigma_d(\bsy)
  \end{split}
  \end{align}
where $K_v(\bsx, \bsy)$ is a reproducing kernel function defined by
\begin{align}\label{eq:defkv}
K_v(\bsx, \bsy) = \int_{-1}^1 v(t)\int_{\sd}\one_{C(\bsz;t)}(\bsx)\one_{C(\bsz;t)}(\bsy)\rd \sigma_d(\bsz)\rd t.
\end{align}
\end{theorem}

\begin{proof}
See Theorem 5.1 in \cite{brau:dick:2013}
\end{proof}

If we set $v(t) = 1$ and $K(\bsx, \bsy) = 1 - C_d\|\bsx-\bsy\|$, then we recover the original Stolarsky formula. Note that the statement of Theorem 5.1 in \cite{brau:dick:2013} has a sign error in their counterpart to \eqref{eq:brauthm}. The corrected statement~\eqref{eq:brauthm} can be verified by comparing equations (5.3) and (5.4) of \cite{brau:dick:2013}.

We would like a version of~\eqref{eq:brauthm} just for one value
of $t$ such as $t=\hat\rho=\bsx_0^\tran\bsy_0$.
For $\hat\rho\in[-1,1)$ and $\beps=(\epsilon_1,\epsilon_2)\in (0,1)^2$, let
\begin{align}\label{eq:veps}
v_{\beps}(t) 
= \epsilon_2 + \frac{1}{\epsilon_1} \one (\hat{\rho}\le t \le\hat{\rho} + \epsilon_1).
\end{align}
Each $v_{\beps}$ satisfies the conditions of Theorem~\ref{thm:weighted-stolarsky}
making~\eqref{eq:brauthm} an identity in $\beps$.
We  let $\epsilon_2\to0$ and then $\epsilon_1\to0$ on both sides of~\eqref{eq:brauthm}
for $v = v_{\beps}$ yielding Theorem~\ref{thm:model1viastolarsky}.

\begin{theorem}\label{thm:model1viastolarsky}
Let $\bsx_0,\bsx_1,\dots,\bsx_{N-1}\in\sd$ and $t\in[-1,1]$. Then
\begin{align}\label{eq:phat1viastolarsky}
 \int_{\sd} |p(\bsy, t) -\hat{p}_1(t)|^2\rd \sigma_d(\bsy) =
\frac{1}{N^2} \sum_{k=0}^{N-1} \sum_{\ell=0}^{N-1}
\sigma_d(C_2(\bsx_k,\bsx_\ell;t))- \hat{p}_1(t)^2.
\end{align}
\end{theorem}
\begin{proof}
See Section~\ref{sec:proof:thm:model1viastolarsky} of the Appendix
which uses the limit argument described above.
\end{proof}

We now give a proposition that holds for all distributions of $\bsy\in\sd$
including our reference distributions~\ref{mod:one} and~\ref{mod:two}.

\begin{proposition}\label{prop:moment}
For a random point $\bsy\in\sd$,
\begin{align}
\e(p(\bsy,t)) &= \frac1N\sum_{k=0}^{N-1}\Pr( \bsy\in C(\bsx_k;t)),\quad\text{and}\label{eq:moment1}\\
\e(p(\bsy,t)^2) &= \frac1{N^2}\sum_{k,\ell=0}^{N-1}\Pr( \bsy\in C_2(\bsx_k,\bsx_\ell;t)).\label{eq:moment2}
\end{align}
\end{proposition}


\begin{proposition}\label{prop:unbiased}
For any $\bsx_0,\dots,\bsx_{N-1}\in\sd$ and $t\in[-1,1]$,
$\hat p_1(t) = \e_1( p(\bsy,t))$.
\end{proposition}
\begin{proof}
    $\e_1( p(\bsy, t)) =
    \e_1\Bigl[\dfrac{1}{N}\sum\limits_{k = 0}^{N -1}
   \one_{C(\bsy;t)}(\bsx_k) \Bigr]
 =\sigma_d(C_d(\bsy; t)) = \hat p_1(t)$.
\end{proof}

Combining Propositions~\ref{prop:moment} and \ref{prop:unbiased} with
Theorem~\ref{thm:model1viastolarsky} we find that
if $\bsy\sim\dustd(\sd)$, as it would for IID Gaussian $Y_i$,
then $p(\bsy,\hat\rho)$ is a random
variable with mean $\hat p_1(\hat\rho)$ and variance given by~\eqref{eq:phat1viastolarsky}
with $t=\hat\rho$. Here $\hat\rho =\bsx_0^\tran\bsy_0$ is fixed
while $\bsy$ is random.

The right hand side of~\eqref{eq:phat1viastolarsky} sums $O(N^2)$
terms. In a permutation analysis we might have $N=n!$ or
$N = {m_0+m_1\choose m_0}$ for binary $X_i$, and so the
computational cost could be high.
The symmetry in a permutation set allows us to use
\begin{align*}
 \int_{\sd} |p(\bsy, t) -\hat{p}_1(t)|^2\rd \sigma_d(\bsy) =
\frac{1}{N} \sum_{k=0}^{N-1} \sigma_d(C_2(\bsx_0,\bsx_k;t))- \hat{p}_1(t)^2
\end{align*}
instead.  This expression costs $O(N)$, the same as the full permutation analysis.
The cost can be reduced for binary $X_i$.

When the $X_i$ are binary, then for fixed $t$,
$\sigma_d(C_2(\bsx_k,\bsx_\ell;t))$
just depends on the swap distance $r$
between $\bsx_k$ and $\bsx_\ell$. Then
\begin{align}\label{eq:phat1viastolarsky3}
 \int_{\sd} |p(\bsy, t) -\hat{p}_1(t)|^2\rd \sigma_d(\bsy) =
\frac{1 }{N} \sum_{r=0}^{\minm}
N_rV_2(u(r);t,d)- \hat{p}_1(t)^2
\end{align}
for $V_2(u(r);t,d)$ given in Lemma~\ref{lem:volintersection},
where $N_r=\sum_{k=0}^{N-1} \sum_{\ell=0}^{N-1} \one(r_{k,\ell}=r)$ counts
pairs $(\bsx_k,\bsx_\ell)$ at swap distance $r$.

\begin{theorem}  \label{thm:phat1}
Let $\bsx_0\in\sd$ be the centered and scaled vector from
an experiment with binary $X_i$ of which $m_0$ are negative and
$m_1$ are positive.
Let $\bsx_0,\bsx_1,\dots,\bsx_{N-1}$
be the $N={m_0+m_1\choose m_0}$ distinct permutations of $\bsx_0$.
If $\bsy\sim\dustd(\sd)$, then for $t\in[-1,1]$, and with $u(r)$ defined in \eqref{eq:ur},
\begin{align*}
\e(p(\bsy,t)) &= \sigma_d(C(\bsy_0;t)),\quad\text{and}\\
\var(p(\bsy, t)) &= \frac1N\sum\limits_{r = 0}^{\minm}
{m_0 \choose r}{m_1 \choose r}V_2(u(r);t,d) - \hat p_1(t)^2.
  \end{align*}

\end{theorem}
\begin{proof}
There are ${m_0\choose r}{m_1\choose r}$
permuted points $\bsx_i$ at swap distance $r$ from $\bsx_0$.
\end{proof}

\section{A finer approximation to the $p$-value}\label{sec:finerapprox}

In the previous section, we studied the distribution of permutation
$p$-values $p(\bsy,t)$ with spherical cap centers $\bsy\sim\dustd(\sd)$
and heights $t=\hat\rho$. 
In this section, we use reference distribution $2$ to obtain
 a finer approximation to $p(\bsy_0,\hat\rho)$
by studying the distribution of the $p$-values with
centers $\bsy$  satisfying the constraint 
$\left<\bsy, \bsx_0 \right> = \left<\bsy_0,
  \bsx_0 \right> = \hat\rho$.
That is $\bsy$ has reference distribution $2$, which
is $\dustd(\bby_2)$.

Our methods also let us impose the constraint
$\left<\bsy, \bsx_c \right> = \left<\bsy_0,
  \bsx_c \right> \equiv \tilde{\rho}$ for any
$c=0,1,2,\dots,N-1$ that we like.
Conditioning on $\langle\bsy,\bsx_c\rangle =\tilde\rho$
eliminates many irrelevant $\bsy$ from consideration.
In addition to our estimate $\hat p_2$ obtained
by $\bsx_c=\bsx_0$, we consider a second choice.
It is to choose $\bsx_c$ to be the
closest permutation of $\bsx_0$ to $\bsy_0$. 
That is $c = {\arg \max_k}\left<\bsy_0, \bsx_k \right>$.


For an index $c\in\{0,1,\dots,N-1\}$
conditioning as above leads to
\begin{align}\label{eq:ptilde}
\tilde p_c =
\e_2( p(\bsy,\hat\rho) )
=
\e_1( p(\bsy,\hat\rho)\mid \bsy^\tran\bsx_c = \bsy_0^\tran\bsx_c),
\end{align}
and our two special cases are
\begin{align}\label{eq:phat2}
\hat{p}_2 & \equiv \tilde p_0,\quad\text{and}\quad
\hat{p}_3  \equiv \tilde p_c,\quad\text{where}\quad c=\arg\max_{0\le k<N}\langle\bsy_0,\bsx_k\rangle.
\end{align}
For an illustration of reference distribution~\ref{mod:two} see Figure~\ref{fig:mod2}.

Notice that $\hat p_2$ cannot go below $1/N$ because all of the points $\bsy$ that
it includes have $\bsx_0\in C(\bsy;\hat\rho)$. In fact $\bsx_0$ is on the boundary
of this spherical cap. Since the true value satisfies $p\ge1/N$, having $\hat p_2\ge1/N$
is a desirable property.
Similarly, $\hat p_3\ge1/N$ because then $\bsx_c$ is in general an interior
point of $C(\bsy,\hat\rho)$. We expect that $\hat p_3$ should
be more conservative than $\hat p_2$ and we see this numerically in
Section~\ref{sec:experimental}.

\begin{figure}[t]
  \centering
  \includegraphics[scale = 0.35]{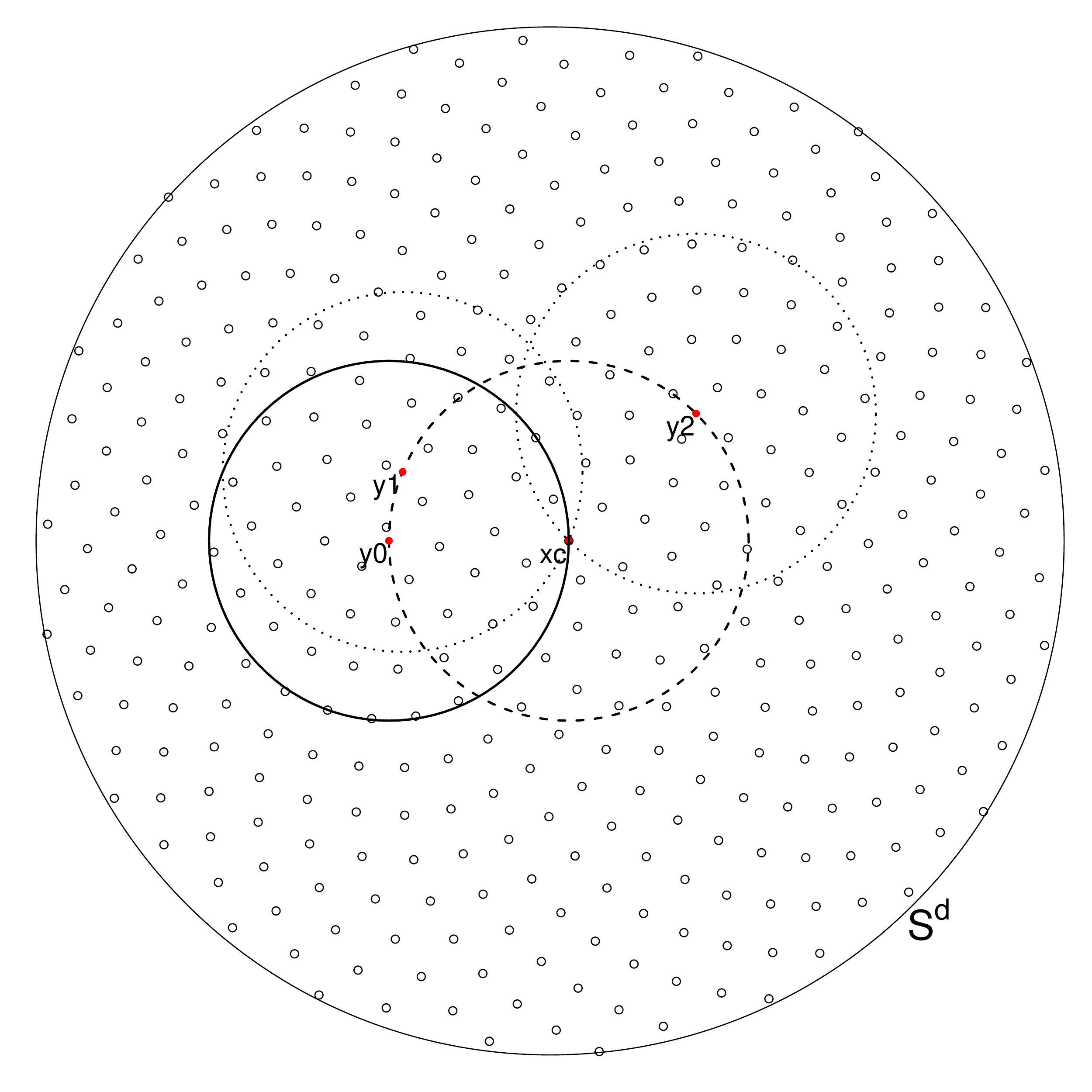}
  \caption{\label{fig:mod2}
Illustration for reference distribution~\ref{mod:two}.
The original response vector is $\bsy_0$ with $\bsy_0^\tran\bsx_0=\hat\rho$
and $\bsx_0$ marked $\bold{xc}$.
We consider alternative $\bsy$ uniformly distributed on the surface of $C(\bsx_0;\hat\rho)$ (dashed circle)
with examples $\bsy_1$ and $\bsy_2$. Around each such $\bsy_j$ there is a spherical cap
of height $\hat\rho$ that just barely includes $\bsx_c=\bsx_0$.
The small open circles are permuted points $\bsx_k$.
The fraction of those open circles that belong to the spherical cap $C(\bsy,\hat\rho)$
is $p(\bsy,\hat\rho)$. We use $\hat p_2 = \e_2(p(\bsy,\hat\rho))$
and find an expression for $\e_2( (\hat p_2-p(\bsy,\hat\rho))^2)$.
}
\end{figure}


From Proposition \ref{prop:moment}, we can get our estimate $\tilde p_c$ and its
mean squared error by finding single and double inclusion probabilties for $\bsy$.

To compute $\tilde p_c$ we need to sum $N$
values $\Pr( \bsy\in C(\bsx_k;t)\mid \bsy^\tran\bsx_c=\tilde\rho)$
and for $\tilde p_c$ to be useful we must compute it in
$o(N)$ time.
The computations are feasible in the binary case
($X_i$ at two levels), which we now focus on.

Let $u_j = \bsx_j^{\tran}\bsx_0$ for $j = 1, 2$, and let $u_3 = \bsx_1^{\tran}\bsx_2$.
Let the projection of $\bsy$ on $\bsx_c$ be $\bsy=\tilde\rho\bsx_c+\sqrt{1-\tilde\rho^2}\bsy^*$.
Then the single and double point inclusion probabilities under reference
distribution~\ref{mod:two} are
\begin{align}\label{eq:P1def}
P_1(u_1, \tilde{\rho}, \hat{\rho})
&= \int_{\sdmo} \one(\left<\bsy, \bsx_{1} \right>
\ge \hat{\rho})\rd \sigma_{d-1}(\bsy^{*}),\quad\text{and} \\
P_2(u_1, u_2, u_3,\tilde{\rho},\hat{\rho}) &=
\int_{\sdmo} \one(\left<\bsy, \bsx_1 \right>
\ge \hat{\rho})\one(\left<\bsy, \bsx_2 \right>
\ge \hat{\rho})\rd \sigma_{d-1}(\bsy^{*}) \label{eq:P2}
\end{align}
where $\hat\rho=\langle \bsx_0,\bsy_0\rangle$.
If two permutations of $\bsx_0$ are at swap distance
$r$, then their inner product is $u(r)=1-r(m_0^{-1}+m_1^{-1})$ from equation~\eqref{eq:ur}.

\begin{lemma}\label{lem:P1}
Let the projection of $\bsx_1$ onto $\bsx_c$
be $\bsx_1 = u_1\bsx_c +\sqrt{1-u_1^2}\bsx_1^*$. Then the single point inclusion probability from~\eqref{eq:P1def} is
\begin{align}\label{eq:P1}
P_1(u_1,\tilde\rho,\hat\rho)=
\begin{cases}
\one(\tilde{\rho}u_1 \ge \hat{\rho}), & u_1 = \pm 1 \text{ or }
\tilde{\rho} = \pm 1\\
\sigma_{d-1}(C(\bsx_1^{*},\rho^*)), & u_1  \in (-1,1), \tilde{\rho} \in (-1 ,1)
\end{cases}
\end{align}
where $\rho^* = (\hat{\rho}-\tilde{\rho}u_1)/{\sqrt{(1-\tilde{\rho}^2)(1-u_1^2)}}$.
\end{lemma}

\begin{proof}
The projection of $\bsy$ onto $\bsx_c$ is
$\bsy = \tilde{\rho}\bsx_c + \sqrt{1 - \tilde{\rho}^2}\bsy^{*}$. 
Now
\begin{equation*}
\left<\bsy, \bsx_1\right>=
  \begin{cases}
\tilde{\rho} u_1, & u_1 = \pm 1 \text{ or }
\tilde{\rho} = \pm 1\\
\tilde{\rho}u_1 +
\sqrt{1-\tilde{\rho}^2}\sqrt{1-u_1^2}\left<\bsy^{*},\bsx_1^{*} \right>,
& u_1 \in (-1,1), \tilde{\rho} \in (-1 ,1)
  \end{cases}
\end{equation*}
and the result easily follows.
\end{proof}


We can now give a computable expression for $\tilde p_c$
 and hence for $\hat p_2$ and $\hat p_3$.

\begin{theorem}\label{thm:defphat2}
For $-1 \le \hat\rho \le 1$ and $ -1 \le \tilde{\rho} \le 1$,
  \begin{align}    \label{eq:phatc}
\tilde p_c =\e_2(p(\bsy, \hat{\rho})) = \frac{1}{N}\sum\limits_{r = 0}^{\minm}
{m_0 \choose r}{m_1 \choose r}P_1(u(r), \tilde{\rho}, \hat{\rho})
  \end{align}
where $u(r)$ is given in equation \eqref{eq:ur},  $P_1(u(r),
\tilde{\rho}, \hat{\rho})$ is given in
equation \eqref{eq:P1} and $\tilde\rho = \bsx_c^\tran\bsy_0$.
\end{theorem}
\begin{proof}
There are ${m_0 \choose r}{m_1 \choose r}$ permutations of $\bsx_0$ at swap distance $r$ from $\bsx_c$.
\end{proof}

From~\eqref{eq:phatc} we see that $\tilde p_c$ can be computed
in $O(\minm)$ work.
The mean squared error for $\tilde p_c$ is more complicated and will be more expensive.
We need the double point inclusion probabilities and then we
need to count the number of pairs $\bsx_k,\bsx_\ell$ forming
a given set of swap distances among $\bsx_k,\bsx_\ell,\bsx_c$.

\begin{lemma}\label{lem:P2}
For $j=1,2$, let $r_j$ be the swap distance of
$\bsx_j$ from $\bsx_c$
and let $r_3$ be the swap distance between $\bsx_1$ and $\bsx_2$.
Let $u_1,u_2,u_3$ be the corresponding inner products given
by~\eqref{eq:ur}.
If there are equalities among $\bsx_1,$ $\bsx_2$ and $\bsx_c$,
then the double point inclusion probability from~\eqref{eq:P2} is
\begin{align*}
P_2(u_1, u_2, u_3,\tilde{\rho},\hat{\rho}) =
  \begin{cases}
\one(\tilde{\rho} \ge \hat{\rho}), &\bsx_1 = \bsx_2 = \bsx_c\\
\one(\tilde{\rho}\ge \hat{\rho}) P_1(u_2, \tilde{\rho}, \hat{\rho}), &\bsx_1 = \bsx_c \ne \bsx_2 \\
\one(\tilde{\rho}\ge \hat{\rho}) P_1(u_1, \tilde{\rho}, \hat{\rho}), &\bsx_2 = \bsx_c \ne \bsx_1 \\
 P_1(u_2, \tilde{\rho}, \hat{\rho}), &\bsx_1 = \bsx_2 \neq \bsx_c.
\end{cases}
\end{align*}
If $\bsx_1$, $\bsx_2$ and $\bsx_c$ are three distinct points
with $\min(u_1,u_2) = -1$, then
$$
P_2(u_1, u_2, u_3,\tilde{\rho},\hat{\rho}) =
\begin{cases}
\one(-\tilde{\rho}\ge \hat{\rho}) P_1(u_2, \tilde{\rho}, \hat{\rho}), & u_1=-1\\
\one(-\tilde{\rho}\ge \hat{\rho}) P_1(u_1, \tilde{\rho}, \hat{\rho}), & u_2=-1.
\end{cases}
$$
Otherwise $-1< u_1, u_2<1$, and then
\begin{align*}
&P_2(u_1, u_2, u_3,\tilde{\rho},\hat{\rho}) \\
&=\begin{cases}
\one(\tilde{\rho}u_1 \ge \hat{\rho})\one(\tilde{\rho}u_2 \ge
\hat{\rho}), &\tilde{\rho} = \pm 1\\[1.2ex]
\int_{-1}^1 \frac{\omega_{d-2}}{\omega_{d-1}}(1-t^2)^{\frac{d-1}{2} - 1} \one(t \ge \rho_1)  \one(tu_3^{*} \ge \rho_2) \rd t,
& \tilde{\rho} \neq \pm 1, u_3^{*} = \pm 1\\[1.2ex]
\int_{-1}^1 \frac{\omega_{d-2}}{\omega_{d-1}}(1-t^2)^{\frac{d-1}{2} - 1} \one(t \ge \rho_1)
\sigma_{d-2}\Bigl(C\bigl(\bsx_2^{**}, \frac{\rho_2 -
  tu_3^{*}}{\sqrt{1-t^2}\sqrt{1-u_3^{*2}}}\bigr)\Bigr) \rd t, & \tilde{\rho}
\neq \pm 1, |u_3^{*}| < 1
\end{cases}
\end{align*}
where
\begin{align}\label{eq:rho12}
u_3^{*} = \frac{u_3 - u_1u_2}{\sqrt{1-u_1^2}\sqrt{1-u_2^2}} \quad\text{and}\quad
\rho_j = \frac{\hat{\rho}  -\tilde{\rho}u_j}{\sqrt{1-\tilde{\rho}^2}\sqrt{1-u_j^2}},\ j=1,2
\end{align}
and $\bsx_2^{**}$ is the residual from the
projection of $\bsx_2^*$ on $\bsx_1^*$.
\end{lemma}
\begin{proof}
See Section~\ref{sec:proof:lem:P2}.
\end{proof}

Next we consider the swap configuration among $\bsx_1$, $\bsx_2$ and $\bsx_c$.
Let $\bsx_j$ be at swap distance $r_j$ from $\bsx_c$, for $j=1,2$.
We let $\delta_1$ be the number of positive components of $\bsx_c$ that are
negative in both $\bsx_1$ and $\bsx_2$. Similarly, $\delta_2$ is the number of
negative components of $\bsx_c$ that are positive in both $\bsx_1$ and $\bsx_2$.
See Figure~\ref{fig:deltas}.
The swap distance between $\bsx_1$ and $\bsx_2$ is
then $r_3 = r_1 + r_2 - \delta_1 - \delta_2$.

\begin{figure}\centering
\begin{align*}
  \bsx_c& = (\,\overbrace{+, +,+, +, +, \cdots, +, +, +, + }^{m_1},\
  \overbrace{ -, -, -, -, \cdots, - ,-, -, -, -}^{m_0}\,)\\
  \bsx_1& = (\,\overbrace{+ , +, +, \cdots, + , \underbrace{- , -,\minuses}_{{r_1}}}^{m_1},\
  \overbrace{\underbrace{\plusess}_{r_1}, \minuses}^{m_0}\,)\\
  \bsx_2& = (\,\overbrace{\pluses, \underbrace{-,-, \underbrace{\cdots, -}_{\delta_1}}_{{r_2}}, \pluses}^{m_1},\
  \overbrace{\minuses, \underbrace{\underbrace{+, +, \cdots}_{\delta_2},
      +}_{r_2}, \minuses}^{m_0}\,)
\end{align*}
\caption{\label{fig:deltas}
Illustration of $r_1$, $r_2$, $\delta_1$ and $\delta_2$.
The points $\bsx_c$, $\bsx_1$ and $\bsx_2$
each have $m_0$ negative and $m_1$ positive components.
For $j=1,2$ the swap distance between $\bsx_j$ and $\bsx_c$ is $r_j$.
There are
$\delta_1$ positive components of $\bsx_c$ where both $\bsx_1$ and
$\bsx_2$ are negative, and $\delta_2$ negative components of $\bsx_c$
where both $\bsx_j$ are positive.
}
\end{figure}

Let $\bsr = (r_1, r_2)$, $\bsdel = (\delta_1, \delta_2)$
and $\minr = \min(r_1,r_2)$.
We will study values of $r_1,r_2,r_3,\delta_1,\delta_2$ ranging over
the following sets:
\begin{align*}
r_1, r_2 & \in R = \{1,\dots,\minm\}\\
\delta_1 & \in D_1(\bsr) = \{\max(0,r_1+r_2-m_0),\dots,\minr\}\\
\delta_2 & \in D_2(\bsr) = \{\max(0,r_1+r_2-m_1),\dots,\minr\},\quad\text{and}\\
r_3 &\in R_3(\bsr) =\{\max(1,r_1+r_2-2\minr),\dots,\min(r_1+r_2,\minm,m_0+m_1-r_1-r_2)\}.
\end{align*}
Whenever the lower bound for one of these sets exceeds the
upper bound, we take the set to be empty, and a sum over it to be zero.
Note that while $r_1=0$ is possible, it corresponds to $\bsx_1=\bsx_c$
and we will handle that case specially, excluding it from $R$.

The number of pairs $(\bsx_\ell,\bsx_k)$ with a fixed  $\bsr$ and $\bsdel$ is
\begin{align}\label{eq:coeff}
c(\bsr, \bsdel) =  {m_0 \choose \delta_1}{m_1 \choose \delta_2}{m_0 - \delta_1 \choose r_1 - \delta_1}{m_1
- \delta_2 \choose r_1 - \delta_2}{m_0 - r_1 \choose r_2 - \delta_1}{m_1 - r_1
\choose r_2 - \delta_2}.
\end{align}
Then the number of configurations given $r_1$, $r_2$ and $r_3$ is
\begin{align}\label{eq:coeffsummed}
c(r_1,r_2,r_3)
= \sum_{\delta_1\in D_1}\sum_{\delta_2\in D_2} c(\bsr,\bsdel)\one( r_3 = r_1+r_2-\delta_1-\delta_2).
\end{align}

We can now get an expression for the expected mean squared error under
reference distribution~\ref{mod:two} which combined with Theorem~\ref{thm:defphat2}
for the mean provides an expression for the mean squared error of $\tilde p_c$.
\begin{theorem}\label{thm:phat2}
For $-1 \le \hat\rho \le 1$ and $-1 \le \tilde{\rho} \le 1$,
\begin{equation}    \label{eq:phat2-second}
\begin{aligned}
\e_2(p(\bsy, \hat{\rho})^2) &= \frac{1}{N^2}\bigg[\one( \tilde{\rho} \ge \hat{\rho}) +
2\sum\limits_{r = 1}^{\minm} {m_0 \choose r}{m_1 \choose r}P_2(1, u(r), u(r),\tilde{\rho}, \hat{\rho})  \\
& + \sum\limits_{r = 1}^{\minm} {m_0 \choose r}{m_1\choose r}P_1(u(r),\tilde{\rho}, \hat{\rho}) \\
& +\sum\limits_{r_1 \in R}\sum\limits_{r_2 \in R}
\sum\limits_{r_3 \in  R_3(\bsr)} c(r_1,r_2,r_3)
P_2(u_1,u_2,u_3,\tilde{\rho}, \hat{\rho})\bigg]
\end{aligned}
  \end{equation}
where $P_2(\cdot)$ is the double inclusion probability in~\eqref{eq:P2} and $c(r_1,r_2,r_3)$
is the configuration count in~\eqref{eq:coeffsummed}.
\end{theorem}

\begin{proof}
See Section~\ref{sec:proof:thm:phat2} of the Appendix.
\end{proof}

In our experience, the cost of computing $\e_2(p(\bsy,\hat\rho)^2)$ under 
reference distribution~\ref{mod:two}
is dominated by the cost of the $O(\minm^3)$ integrals required to get
the $P_2(\cdot)$ values in~\eqref{eq:phat2-second}.
The cost also includes an $O(\minm^4)$ component because $c(r_1,r_2,r_3)$
is also a sum of $O(\minm)$ terms, but it did not dominate the computation
at the sample sizes we looked at (up to several hundred).

\section{Generalized Stolarsky Invariance}\label{sec:genstolarsky}
Here we obtain the results for reference distribution~\ref{mod:two} in a different way,
by extending the work by \cite{brau:dick:2013}.
They introduced a weight on the height $t$ of the spherical cap in the average.
We now apply a weight function to the inner product $\langle\bsz,\bsx_c\rangle$
between the center $\bsz$ of the spherical cap and a special point~$\bsx_c$.

\begin{theorem}\label{thm:location-weighted-stolarsky}
Let $\bsx_0, \dots, \bsx_{N-1}$ be arbitrary points in $\sd$
and $v(\cdot)$ and $h(\cdot)$ be positive functions in $L_2([-1,1])$.
Then for any $\bsx' \in  \sd$, the following equation holds,
  \begin{equation}\label{eq:16}
  \begin{split}
&    \int_{-1}^1v(t) \int_{\sd} h(\left <\bsz, \bsx' \right>)\biggl| \sigma_d(C(\bsz;t)) -
      \frac{1}{N}\sum\limits_{k=0}^{N-1}\one_{C(\bsz;t)}(\bsx_k)\biggr|^2
    \rd \sigma_d(\bsz) \rd t\\
=&\frac{1}{N^2}\sum\limits_{k,\ell=0}^{N-1}\Kv(\bsx_k, \bsx_\ell) + \int_{\sd}\int_{\sd}\Kv(\bsx,\bsy)\rd \sigma_d(\bsx)\rd \sigma_d(\bsy) \\
&- \frac{2}{N}\sum\limits_{k= 0}^{N-1}\int_{\sd}\Kv(\bsx, \bsx_k)\rd \sigma_d(\bsx)
  \end{split}
  \end{equation}
where $\Kv:\sd \times \sd \to \real$ is a reproducing kernel defined by
\begin{equation}\label{eq:generalized-kernel}
\Kv(\bsx, \bsy) = \int_{-1}^1 v(t)\int_{\sd}h(\left<\bsz, \bsx' \right>
)\one_{C(\bsz;t)}(\bsx)\one_{C(\bsz;t)}(\bsy)\rd \sigma_d(\bsz)\rd t.
\end{equation}
\end{theorem}

\begin{proof}
See Section~\ref{sec:proof:thm:location-weighted-stolarsky} of the Appendix.
\end{proof}

\begin{remark}
We will use this result for $\bsx'=\bsx_c$, where $\bsx_c$
is one of the $N$ given points.
The theorem holds for general $\bsx'\in\sd$, but the result is computationally and statistically
more attractive when $\bsx'=\bsx_c$.
\end{remark}

We now show that the second moment in Theorem \ref{thm:phat2} holds as a special limiting case of Theorem \ref{thm:location-weighted-stolarsky}.
In addition to $v_{\beps}$ from Section~\ref{sec:stolarskys}
we introduce $\bseta = (\eta_1,\eta_2)\in(0,1)^2$ and
\begin{align}\label{eq:heta}
h_{\bseta}(s) = \eta_{2} +
\frac{1}{\eta_{1}(\frac{\omega_{d-1}}{\omega_{d}}(1-s^2)^{d/2-1})}\one(\tilde{\rho}\le s \le
\tilde{\rho} + \eta_{1})
\end{align}

Using these results we can now establish the following theorem, which provides the second moment of $p(\bsy, \hat{\rho})$ under reference distribution~\ref{mod:two}.
\begin{theorem}\label{thm:invariancemodeltwosecondmoment}
Let $\bsx_0\in\sd$ be the centered and scaled vector from
an experiment with binary $X_i$ of which $m_0$ are negative and
$m_1$ are positive.
Let $\bsx_0,\bsx_1,\dots,\bsx_{N-1}$
be the $N={m_0+m_1\choose m_0}$ distinct permutations of $\bsx_0$.
Let $\bsx_c$ be one of the $\bsx_k$ and define $\tilde p_c$ by~\eqref{eq:ptilde}.
Then
\begin{align*}
\e_2(\tilde p_c(\bsy, \hat{\rho})^2) 
&= \frac{1}{N^2}\sum\limits_{k,\ell=0}^{N-1}\int_{\sdmo}\one(\left<\bsy, \bsx_k \right>
\ge \hat{\rho})\one(\left<\bsy, \bsx_\ell \right>
\ge \hat{\rho}) \rd \sigma_{d-1}(\bsy^{*})
\end{align*}
where $\bsy = \tilde{\rho}\bsx_c + \sqrt{1 -
  \tilde{\rho}^2}\bsy^{*}$. 
\end{theorem}

\begin{proof}
The proof uses Theorem~\ref{thm:location-weighted-stolarsky} with 
a sequence of $h$ defined in \eqref{eq:heta} and $v$ defined in~\eqref{eq:veps}.
See Section~\ref{sec:proof:thm:invariancemodeltwosecondmoment}
of the appendix.
\end{proof}

This result shows that we can use the invariance principle
to derive the second moment of $p(\bsy,\hat\rho)$ under 
reference distribution~\ref{mod:two}.
The mean square in Theorem~\ref{thm:invariancemodeltwosecondmoment}
is consistent with the second moment equation~\eqref{eq:moment2} in Proposition~\ref{prop:moment}.

\section{Two-sided p-values}\label{sec:two-sided}

In statistical applications it is more usual to report two-sided $p$-values. 
A conservative approach is to use $2\min(p,1-p)$ where $p$ is a one-sided $p$-value.
A sharper choice is
 \begin{equation*}
     p = \frac{1}{N}\sum\limits_{k = 0}^{N-1}\one
       (|\bsx_k^{\tran}\bsy_0| \ge |\hat{\rho}|).
 \end{equation*}
This choice changes our estimate under reference distribution~\ref{mod:two}.
It also changes the second moment of our estimate $\hat p_1$.

The two-sided version of the 
estimate $\hat{p}_1(\hat{\rho})$   is
$2\sigma_d(C(\bsy;|\hat{\rho}|))$, the same as if we had doubled a 
one-sided estimate. Also $\e_1(p)=\hat p_1$ in the two-sided case.
We now consider the mean square for
the two-sided estimate under reference distribution~\ref{mod:one}.
For $\bsx_1, \bsx_2\in\sd$ with $u = \bsx_1^{\tran}\bsx_2$, the two-sided double inclusion probability under
reference distribution~\ref{mod:one} is
\begin{align*}
\tilde{V}_2(u;t,d) = \int_{\sd}\one(|\bsz^{\tran}\bsx_1| \ge |t|)\one(|\bsz^{\tran}
  \bsx_2| \ge |t|)\rd \sigma_d(\bsz).
\end{align*}
Writing $\one (|\bsz^{\tran}\bsx_i|
    \ge |t|) = \one (\bsz^{\tran}\bsx_k
    \ge |t|) + \one (\bsz^{\tran}(-\bsx_k)
    \ge |t|) $ for $k = 1,2$ and expanding the product, we get
$$\tilde{V}_2(u;t, d) = 2V_2(u; |t|, d) + 2V_2(-u; |t|, d).$$
By replacing $V_2(u, t, d)$ with $\tilde{V}_2(u, t, d)$ and
$\hat{p}_1(t)$ with $2\sigma_d(C(\bsy;|t|))$ in Theorem~\ref{thm:phat1}, we get the
variance of two-sided p-values under reference distribution~\ref{mod:one}.

To obtain corresponding formulas under reference distribution~\ref{mod:two}, 
we use the usual
notation. Let $u_j = \bsx_j^{\tran}\bsx_0$ for $j = 1, 2$, and let $u_3 = \bsx_1^{\tran}\bsx_2$.
Let the projection of $\bsy$ on $\bsx_c$ be
$\bsy=\tilde\rho\bsx_c+\sqrt{1-\tilde\rho^2}\bsy^*$.
Now 
\begin{align}\label{eq:P1tildedef}
\tilde{P}_1(u_1, \tilde{\rho}, \hat{\rho})
&= \int_{\sdmo} \one\bigl(|\langle\bsy, \bsx_{1} \rangle|
\ge |\hat{\rho}|\bigr)\rd \sigma_{d-1}(\bsy^{*}),\quad\text{and,} \\
\tilde{P}_2(u_1, u_2, u_3,\tilde{\rho},\hat{\rho}) &=
\int_{\sdmo} \one\bigl(|\langle\bsy, \bsx_1 \rangle|
\ge |\hat{\rho}|\bigr)\one\bigl(|\langle\bsy, \bsx_2 \rangle|
\ge |\hat{\rho}|\bigr)\rd \sigma_{d-1}(\bsy^{*})\label{eq:P2tilde}
\end{align}
are the appropriate single and double inclusion probabilities.

After writing $\one(|\langle\bsy, \bsx_k \rangle|
\ge |\hat{\rho}|) = \one(\langle\bsy, \bsx_k \rangle
\ge |\hat{\rho}|) + \one(\langle\bsy, -\bsx_k \rangle
\ge |\hat{\rho}|)$ for $k = 1,2 $ and expanding the product, we get
\begin{align*}
  \tilde{P}_1(u_1, \tilde{\rho}, \hat{\rho}) &= P_1(u_1, \tilde{\rho},
  |\hat{\rho}|) + P_1(-u_1, \tilde{\rho}, |\hat{\rho}|),\quad\text{and}\\
  \tilde{P}_2(u_1, u_2, u_3, \tilde{\rho}, \hat{\rho}) &= P_2(u_1,
  u_2, u_3, \tilde{\rho}, |\hat{\rho}|) + P_2(-u_1,
  u_2, -u_3, \tilde{\rho}, |\hat{\rho}|) \\
&+ P_2(u_1,
  -u_2, -u_3, \tilde{\rho}, |\hat{\rho}|) + P_2(-u_1,
  -u_2, u_3, \tilde{\rho}, |\hat{\rho}|).
\end{align*}

Changing  $P_1(u_1, \tilde{\rho}, \hat{\rho})$ and $P_2(u_1, u_2,u_3,
\tilde{\rho}, \hat{\rho})$ to $\tilde{P}_1(u_1, \tilde{\rho}, \hat{\rho})$
and $\tilde{P}_2(u_1, u_2,u_3,\tilde{\rho}, \hat{\rho})$ respectively
in Theorems \ref{thm:defphat2} and \ref{thm:phat2}, we get the
first and second moments for two-sided p-values under
reference distribution~\ref{mod:two}.

For a two-sided $p$-value, $\hat p_3$ is calculated with $\bsx_c$
where $\tilde{c} = {\arg \max_k}|\langle\bsy_0, \bsx_k \rangle|$. For
$m_0 = m_1$, $\tilde{c} = c = {\arg \max_k}\langle\bsy_0, \bsx_k
\rangle$, but the result may differ significantly for unequal sample
sizes.


\section{Numerical Results}\label{sec:experimental}

We consider two-sided $p$-values in this section. 
The main finding is that the root mean squared
error (RMSE) of $\hat p_2$ under reference distribution 2
is usually just a small multiple of $\hat p_2$ itself.

First we evaluate the accuracy of $\hat{p}_1$, the simple spherical
cap volume approximate $p$ value.
We considered $m_0=m_1$ in a range of values from $5$
to $200$. The values $\hat p_1$ ranged from just below $1$
to $2\times10^{-30}$.
We judge the accuracy of this estimate by its RMSE.
Under distribution~\ref{mod:one} this is
$(\e (\hat p_1(\rho)-p(\bsy,\rho))^2)^{1/2}$ for $\bsy\sim\dustd(\sd)$.
Figure~\ref{fig:step1-1} shows this RMSE decreasing towards 0 as
$\hat{p}_1$ goes to 0 with $\rho$ going to 1.  The RMSE also decreases
with increasing sample size, as we would expect from the central limit theorem.

As seen in Figures~\ref{fig:step1-1} and \ref{fig:step1-1-zoom}, the RMSE is not
monotone in $\hat p_1$. Right at $\hat p_1=1$ we know that $\mathrm{RMSE}=0$
and around $0.1$ there is a dip. The practically interesting values of $\hat p_1$
are much smaller than $0.1$, and the RMSE is monotone for them.

A problem with $\hat p_1$ is that it can approach $0$
even though $p\ge 1/N$ must hold.  
The distribution~\ref{mod:one} RMSE
does not reflect this problem.
By studying $\e_2( (\hat p_1(\rho)-p(\bsy,\rho))^2)^{1/2}$,
we get a different result.
In Figure~\ref{fig:step2-phat1-1}, the RMSE
of $\hat{p}_1$ under distribution~\ref{mod:two} reaches a 
plateau as $\hat{p}_1$ goes to 0.

The estimator $\hat p_2 = \tilde p_0$ performs better than $\hat p_1$
because it makes more use of the data, and it is never below $1/N$.
As seen in Figure~\ref{fig:step2-phat2-1},
the RMSE of $\hat p_2$ very closely matches $\hat p_2$ itself
as $\hat p_2$ decreases to zero. That is, the relative error $|\hat p_2-p|/\hat p_2$
is well behaved for small $p$-values.  
In rare event estimation, that property is known as strong efficiency \citep{blanchet2008efficient} and can be very hard to achieve.
Here as $\hat p_2$ decreases to the granularity
limit $1/N$, its RMSE actually decreases to $0$.  Eventually
the distance from $\bsy_0$ to $\bsx_0$ is below
the minimum interpoint distance among the $\bsx_k$
and then, for a one-sided test, $\hat p_2 = p = 1/N$.

The estimators $\hat{p}_1$
and $\hat p_2$, do not differ much for larger $p$-values
as seen in Figure~\ref{fig:step2-phat2-2}.  But in the limit as $\hat\rho\to1$
we see that $\hat p_1\to0$, while $\hat p_2$ approaches the granularity limit $1/N$ instead.

Figure~\ref{fig:step2-phat2-3} compares the RMSE of the two estimators
under distribution~\ref{mod:two}.
As expected, $\hat p_2$ is more accurate.
It also shows that the biggest differences occur only when $\hat{p}_1$ goes
below ${1}/{N}$.

To examine the behavior of $\hat p_2$ more closely, we plot its
coefficient of variation in Figure~\ref{fig:step2-phat2-4}.
We see that the relative uncertainty in $\hat p_2$ is not extremely
large. Even when the estimated $p$-values are as small as
$10^{-30}$ the coefficient of variation is below $5$. 

In Section \ref{sec:finerapprox}, we mentioned
another choice for $\bsx_c$. It was $\hat p_3=\tilde p_c$, where
$\bsx_c$ is the closest permutation  of $\bsx_0$ to $\bsy_0$.
Figure 2.7 in \cite{he:2016} compares $\hat p_3$ to $\hat p_2$
in some simulations.
As expected, $\hat p_3$ tends to be larger (more conservative)
than $\hat p_2$, though it does sometimes come out smaller.
Figure 2.8 of \cite{he:2016} compares the RMSE
of $\hat p_3$ to $\hat p_2$.  The upward bias of $\hat p_3$
gave it a much larger RMSE.

\begin{figure}[t!]
    \centering
    \begin{subfigure}[b]{0.45\textwidth}
        \includegraphics[width=\textwidth]{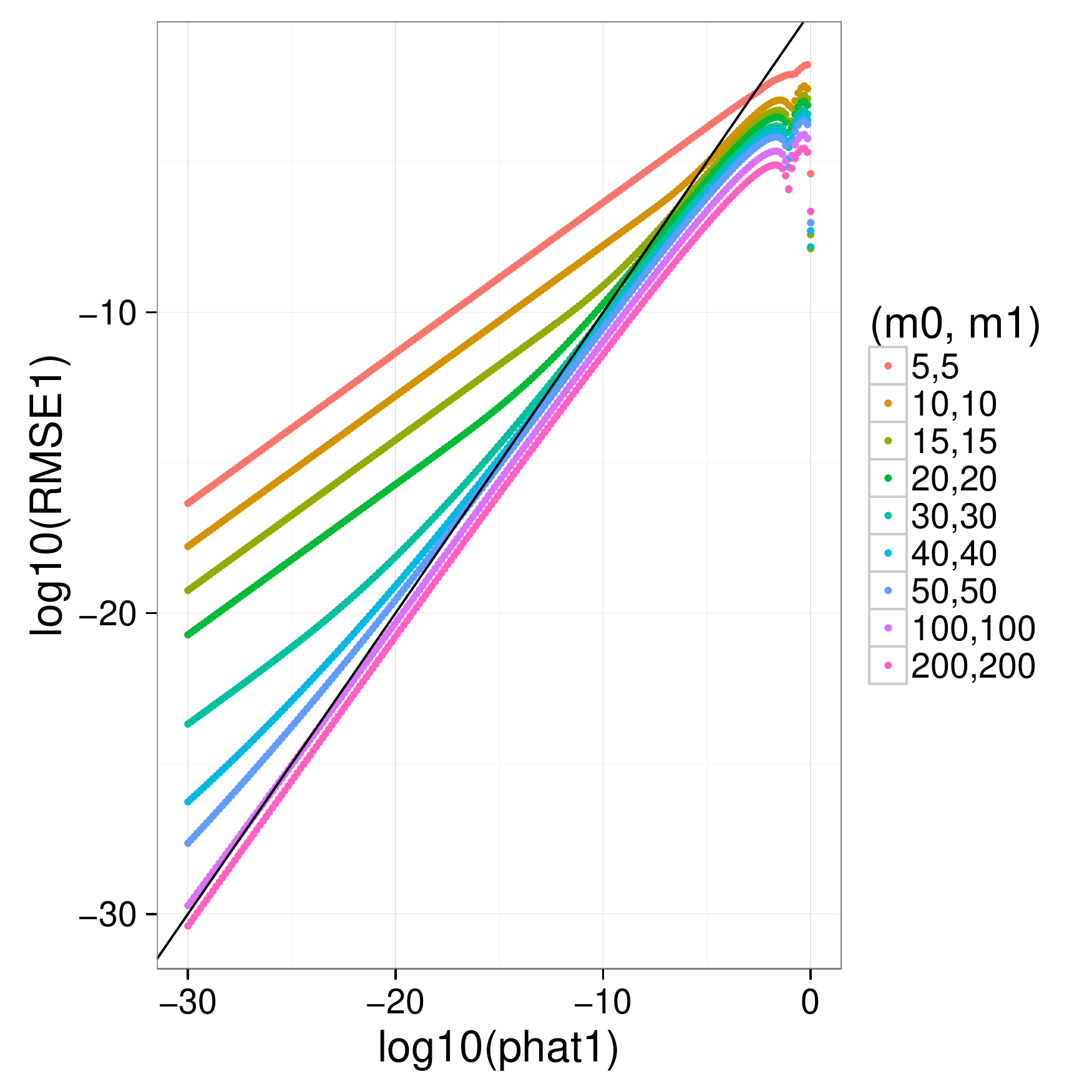}
        \caption{RMSE$_1(\hat{p}_1)$.}
        \label{fig:step1-1}
    \end{subfigure}
~
    \begin{subfigure}[b]{0.45\textwidth}
        \includegraphics[width=\textwidth]{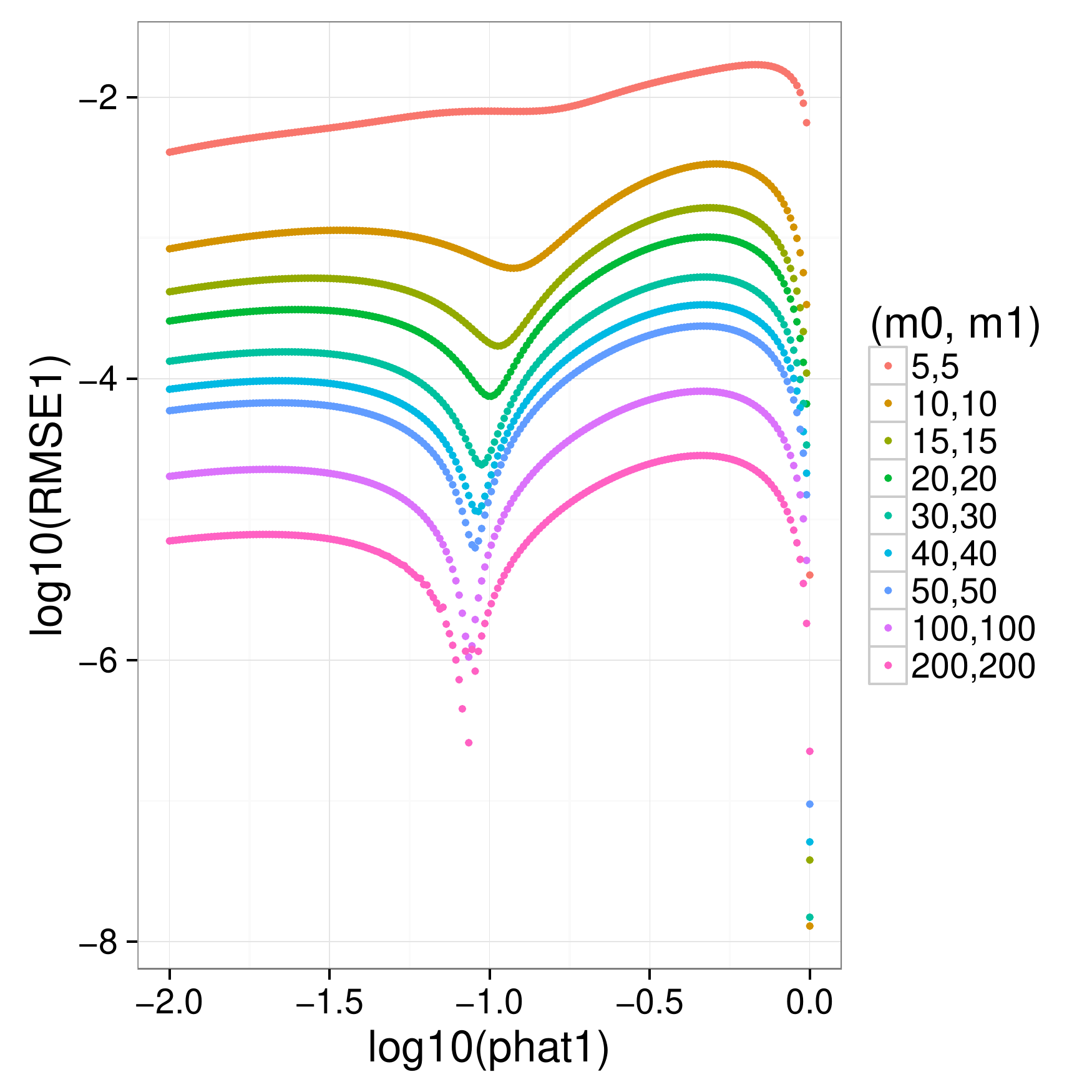}
        \caption{RMSE$_1(\hat{p}_1)$ zoomed.}
        \label{fig:step1-1-zoom}
    \end{subfigure}

    \begin{subfigure}[b]{0.45\textwidth}
        \includegraphics[width=\textwidth]{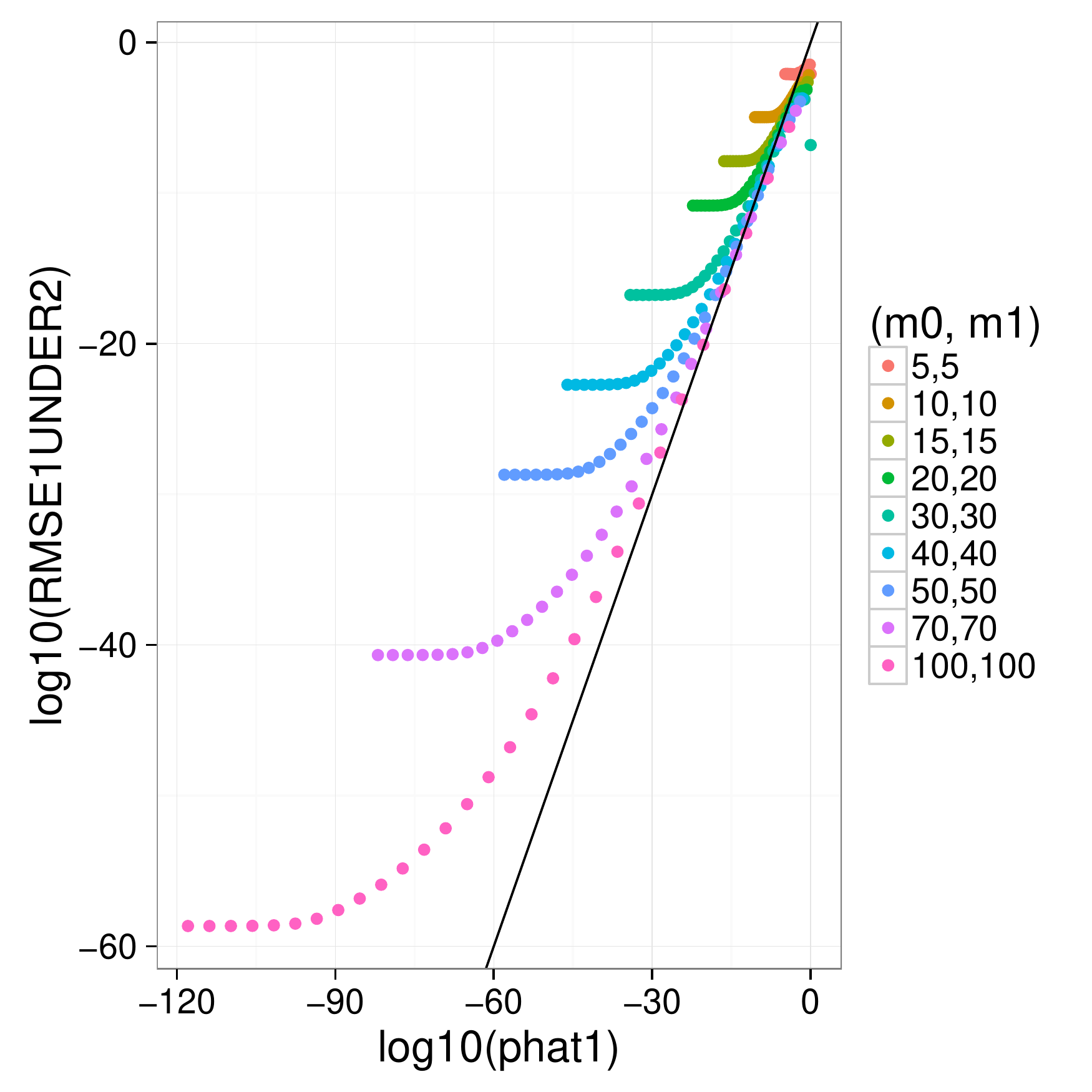}
        \caption{RMSE$_2(\hat{p}_1)$.}
        \label{fig:step2-phat1-1}
    \end{subfigure}
~
    \begin{subfigure}[b]{0.45\textwidth}
        \includegraphics[width=\textwidth]{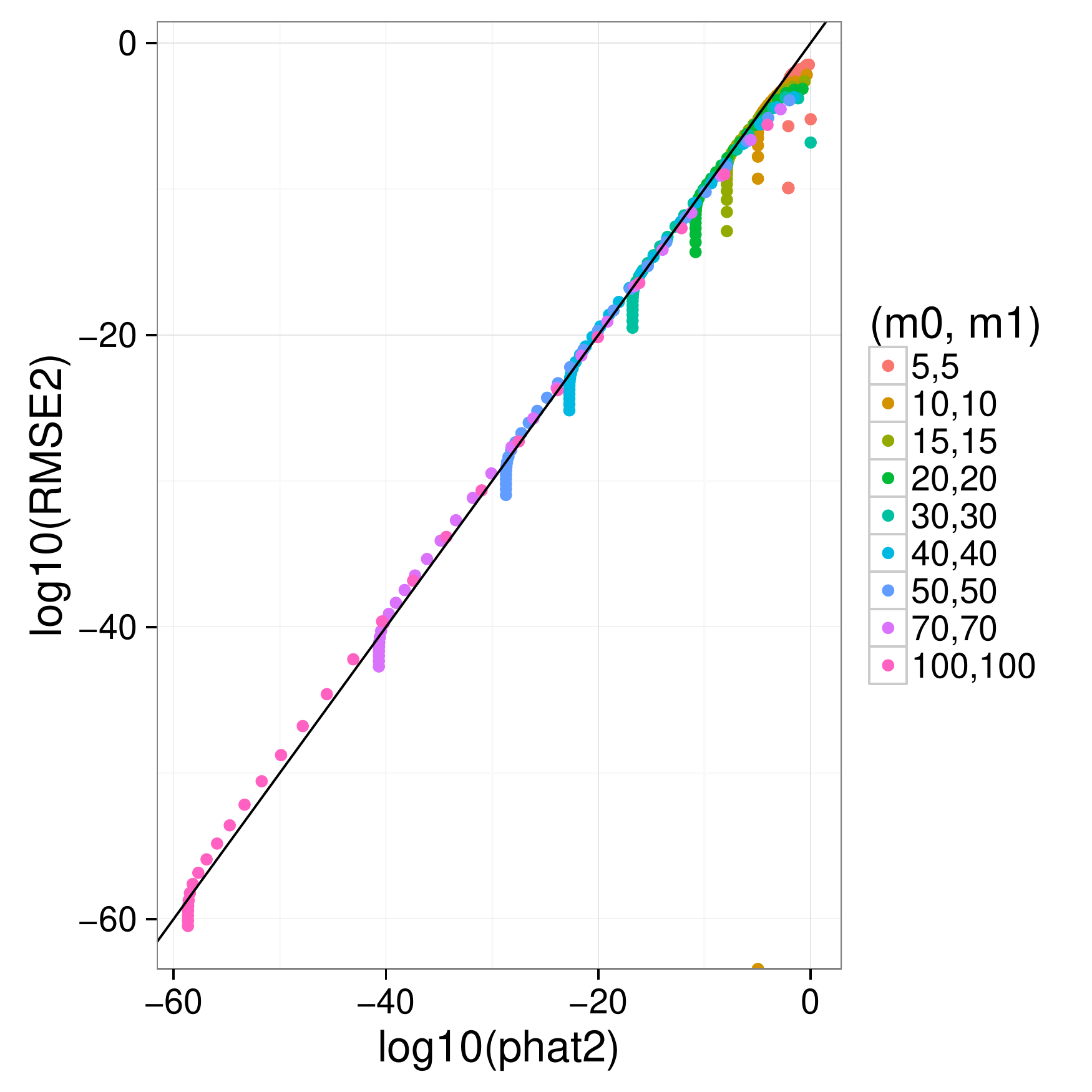}
        \caption{RMSE$_2(\hat{p}_2)$.}
        \label{fig:step2-phat2-1}
    \end{subfigure}

\caption{RMSEs for $\hat{p}_1$ and $\hat{p}_2$ under reference
distributions \ref{mod:one} and \ref{mod:two}.
The $x$-axis shows the estimate $\hat p$
as $\rho$  varies from $1$ to $0$.
Here $m_0 = m_1$. Plots with $m_0 \neq m_1$ are similar.}
\end{figure}

\begin{figure}[t]
  \centering
    \begin{subfigure}{0.45\textwidth}
        \includegraphics[width=\textwidth]{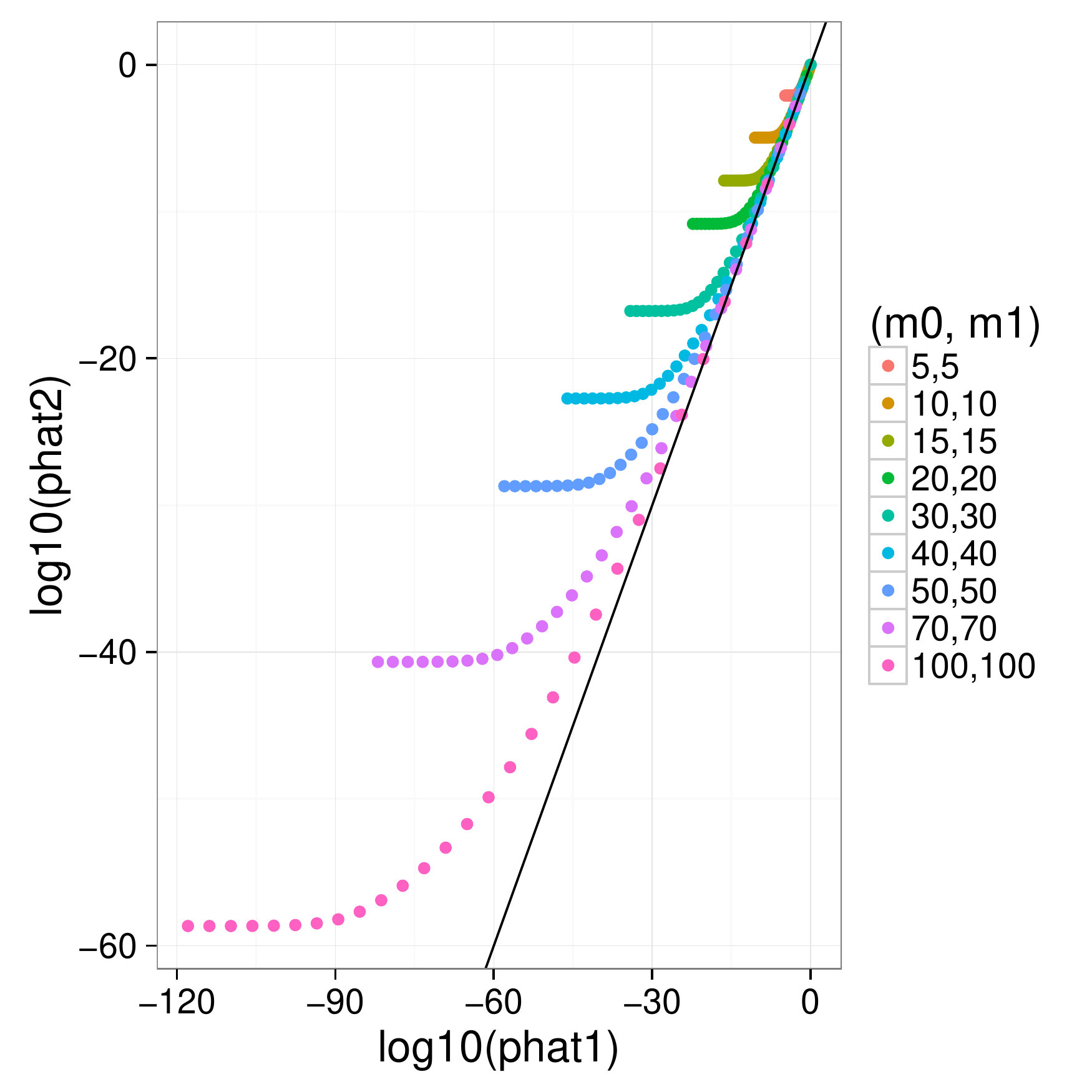}
        \caption{}
        \label{fig:step2-phat2-2}
    \end{subfigure}
~
    \begin{subfigure}{0.45\textwidth}
        \includegraphics[width=\textwidth]{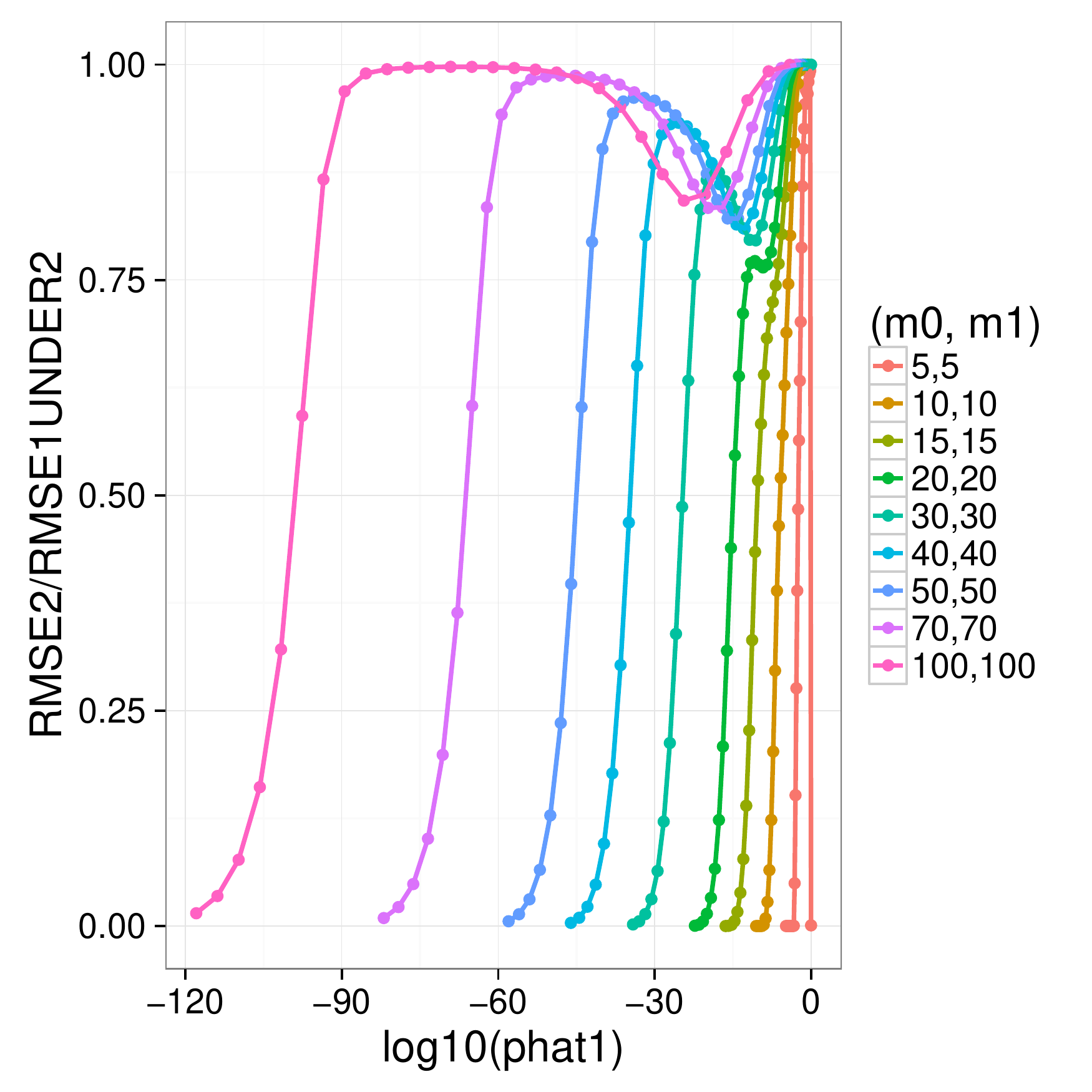}
        \caption{}
        \label{fig:step2-phat2-3}
    \end{subfigure}
  \caption{Comparison of $\hat{p}_1$ and $\hat{p}_2$. In (a), $\log_{10}(\hat{p}_2)$ is plotted
    against $\log_{10}(\hat{p}_1)$ for varying $\rho$'s. The black line
    is the 45 degree line. In (b), the ratio of RMSEs for $\hat{p}_1$
    and $\hat{p}_2$ is plotted against $\log_{10}(\hat{p}_1)$. The
    $x$-axis is $\log_{10}(\hat{p}_1)$.}
\end{figure}

\begin{figure}[t!]
  \centering
  \includegraphics[width = 0.5\textwidth]{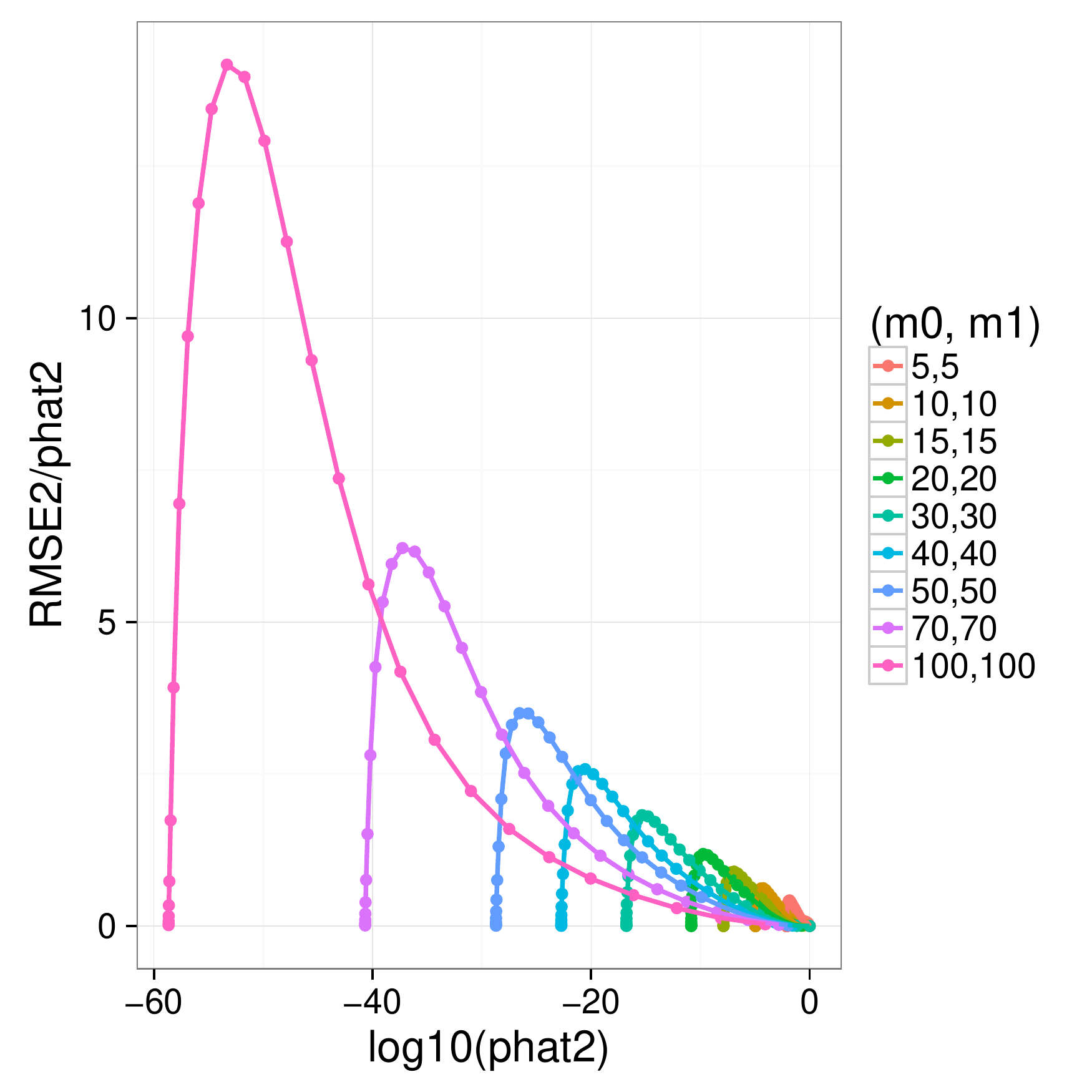}
  \caption{The coefficient of variation for $\hat{p}_2$ with varying $\rho$'s. }
\label{fig:step2-phat2-4}
\end{figure}

\section{Comparison to saddlepoint approximation}\label{sec:saddle}

The small relative error property of $\hat p_2$ is
similar to the relative error property in saddlepoint approximations.
\cite{reid:1988} surveys saddlepoint approximations and \cite{robi:1982} 
develops them for permutation tests of the linear statistics we have considered here.  
When the true $p$-value is $p$, the saddlepoint approximation  
$\hat p_{s}$ satisfies $\hat p_s = p(1+O(1/n))$.  
Because we do not know the implied constant in $O(1/n)$ or
the $n$ at which it takes effect, the saddlepoint approximation does not
provide a computable upper bound for the true permutation $p$-value $p$.  

Figure~\ref{fig:saddle-exp}
compares our estimates to each
other and those of the saddlepoint approximation, equation (1) from \cite{robi:1982}.
The simulated data have the $\dexp(1)$ distribution under the control 
condition and the $2+\dexp(1)$ distribution under the affected condition.
The sample sizes were $m_0=m_1=10$
making it feasible to compute the exact permutation $p$-value for hundreds of 
examples.
In each case we ran $500$ independent simulations. Cases with perfect separation
were excluded: the saddlepoint approximation is numerically unstable then, and one
can easily detect that the minimum $Y$ value in one group is larger than the maximum
in the other group, showing that $p=1/N$. 
In every instance we compared two-sided $p$-values.
Chapter 2 of \cite{he:2016} considers simulations from some other distributions. 
The control condition data are $t_{(5)}$, $\dnorm(0,1)$ and $\dustd(0,1)$
while the affected condition data are shifted versions of these distributions.

In these simulations, the naive spherical cap estimator $\hat p_1$,
with no good relative error properties,
is consistently least accurate and is often much smaller than the 
true $p$.  The saddlepoint estimate is very accurate but tends 
to come out smaller than the true $p$. The estimators 
$\hat p_2$ and $\hat p_3$ are less likely to be below $p$ than 
the saddlepoint estimate, and by construction, they are never
below the granularity limit.
Qualitatively similar results happened for all of the distributions.
The accuracy of all of these
$p$-value estimates tends to be better for ligher tailed $Y_i$.

\begin{figure}[t!]
  \centering
  \includegraphics[width = 0.8\textwidth]{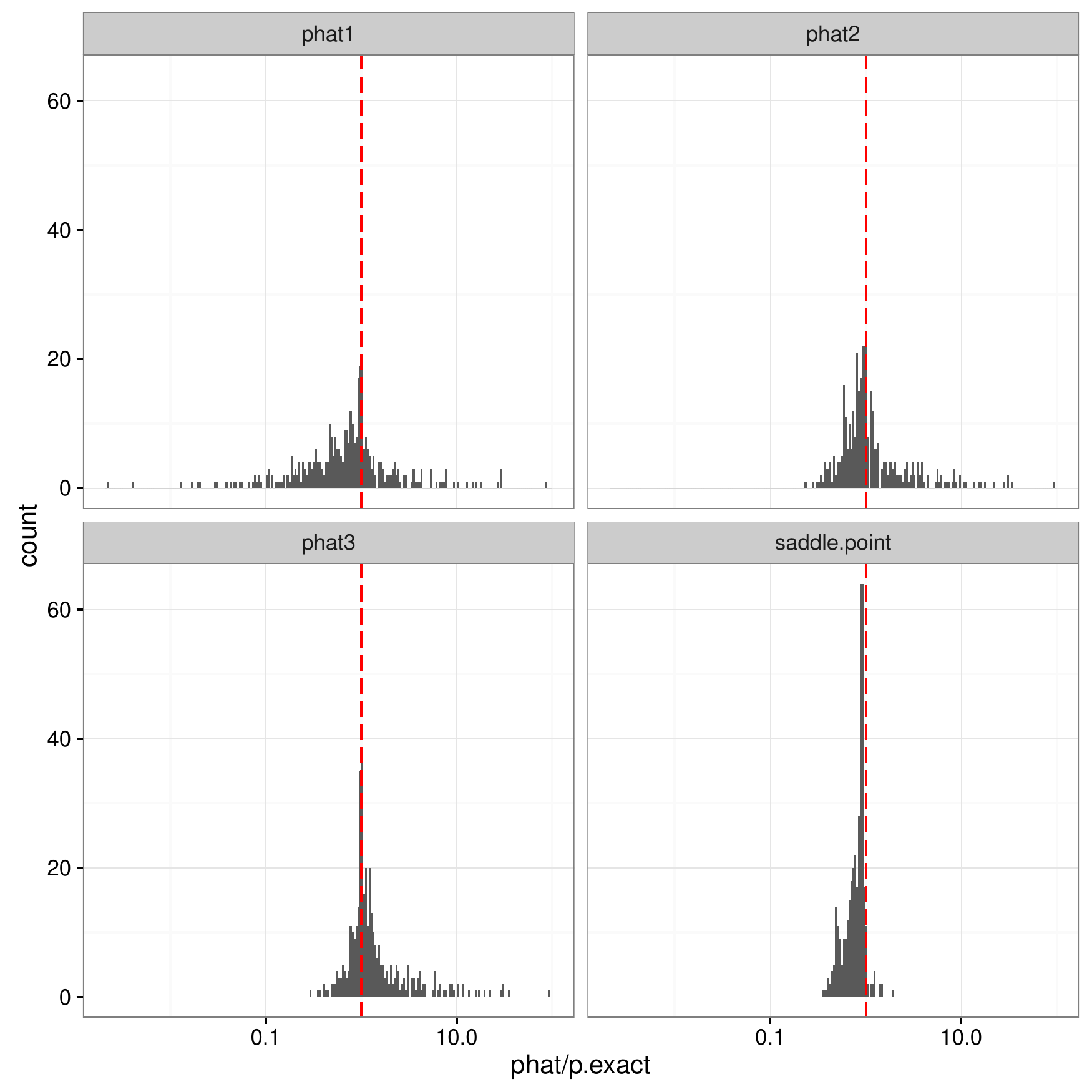}
  \caption{
Simulation results $\hat p/p$ as described in the text,
for $Y_{0,i}\simiid \mathrm{Exp}(1)$, and $Y_{1, i} \simiid \mathrm{Exp}(1)+2$.
 }
\label{fig:saddle-exp}
\end{figure}

We can also construct $Z$ scores, $Z_2=(p-\hat p_2)/\mathrm{RMSE}_2$
and a similar $Z_3$. If these take large values, then it means that $\hat p$
is too small and, moreover, that our computed RMSE does not
diagnose it.
The largest $Z$ scores we observed are in Table~\ref{tab:zvals}. 
The largest $Z$ values arose for exponential data with $p\doteq 0.89$
and $\hat p_2\doteq0.78\doteq\hat p_3$. Such large $p$-values are not very
important and so maximal $Z$ scores are also shown among estimated
$p$-values below~$0.1$.

\begin{table}\centering
\begin{tabular}{lcccc}
Dist'n $Y_{0,i}$ & $\max Z_2$  & $\max \{Z_2\mid \hat p_2<0.1\}$ & $\max Z_3$  & $\max \{Z_3\mid \hat p_3<0.1\}$\\
\midrule
$t_{(5)}$& $26.7$ & $\phz1.91$&  $31.5$ & $\phz3.87$\\
Exp(1) & $\phz\phz7.55$ & $\phz7.55$ & $\phz\phz7.76$ & $\phz7.76$ \\
$\dustd(0,1)$ &$\phz\phz3.49$ & $\phz2.45$ & $\phz\phz5.87$ & $\phz2.61$\\
$\dnorm(0,1)$ & $\phz\phz3.07$ & $\phz2.78$ & $\phz\phz3.07$ & $\phz2.78$\\
\bottomrule
\end{tabular}
\caption{\label{tab:zvals}
Maximal $Z$ scores observed for $\hat p_2$ and $\hat p_3$
in $500$ independent replications.
}
\end{table}

What we find is that the $Z$ values are not very extreme.  This suggests
that it might be feasible to get a conservative $p$-value estimate
by adding some multiple of the distribution 2 RMSE to $\hat p_2$.

\section{Data comparisons}\label{sec:data}

Three data sets on Parkinson's disease were used by \cite{lars:owen:2015}
and investigated in Chapter 6 of \cite{he:2016}.
They come from \cite{scherzer}, \cite{moran2006whole}
and \cite{zhang2}.
Table~\ref{tab:studies} shows their sample sizes.

\begin{table}
\centering
\begin{tabular}{lccc}
First author & $m_1$ &$m_0$ & $N={m_1+m_0\choose m_1}$\\
\midrule
Zhang & 11 & 18 & $3.5\times 10^7\phz$\\
Moran & 29 & 14 & $7.9\times 10^{10}$\\
Scherzer & 50 & 22 & $1.8\times 10^{18}$\\
\bottomrule
\end{tabular}
\caption{\label{tab:studies}
Sample sizes for three microarray studies.
}
\end{table}

For this comparison, there were 6180 gene sets from 
v5.1 of mSigDB's gene set collections. Curated gene sets
and Gene Ontology gene sets were used. The gene sets ranged
in size from $5$ to $2131$ genes with an average size of
$93.08$ genes. Slightly different
versions of the gene sets were used in \cite{lars:owen:2015}.

Ground truth estimates of two-sided  p values for linear test
statistics were obtained using $M=10^6$ permutations. When
the estimate was below $10^{-4}$, the number $M$ was increased
to $10^7$.
The Zhang data set had the smallest sample size and had no
gene sets significant at below $0.01$ and so we do not
compare estimates for this gene set.

Table~\ref{tab:corrs} gives correlations between
$\log_{10}$ estimated and gold-standard $p$-values
for these genes.  
From Table~\ref{tab:corrs}, we see that $\hat p_1$, $\hat p_2$
and $\hat p_3$ have nearly the same correlations with the gold
standard; indeed they correlate highly with each other. They
correlate with the gold standard estimate much more closely
than the saddlepoint estimator does. 
Figures in Chapter $6$ of \cite{he:2016} give scatterplots.  
These show the saddlepoint estimator is biased low
and $\hat p_3$ is biased slightly high. Statistics $\hat p_1$
and $\hat p_2$ are quite close, possibly because none
of the gene sets has a very small $p$-value.

\begin{table}
\centering
\begin{tabular}{lccccccc}
Data & Correlation & $\hat p_1$& $\hat p_2$& $\hat p_3$ & $\hat p_{\text{saddle}}$\\
\midrule 
Moran & Pearson & $0.9997$ & $0.9997$ & $0.9997$ & $0.9937$\\
Moran & Kendall & $0.9856$ & $0.9856$& $0.9866$ & $0.9395$\\
\midrule
Scherzer & Pearson & $0.9997$ & $0.9997$ & $0.9997$ & $0.9837$\\
Scherzer & Kendall & $0.9869$ & $0.9869$& $0.9870$ & $0.8937$\\
\midrule 
Moran Low & Pearson & $0.9501$ & $0.9504$ & $0.9653$ & $0.7125$\\
Moran Low & Kendall & $0.8283$ & $0.8283$ & $0.8578$ & $0.5411$\\
\midrule 
Scherzer Low & Pearson & $0.9940$ & $0.9940$ & $0.9947$ & $0.8652$\\
Scherzer Low & Kendall & $0.9429$ & $0.9429$& $0.9429$ & $0.7714$\\
\bottomrule
\end{tabular}
\caption{\label{tab:corrs}
The table gives both Pearson and Kendall correlations over 6180
gene sets, between estimated $\log_{10}(p)$-values and gold standard 
$\log_{10}(p)$-values 
for both the Moran and Scherzer data sets.  
The designation `Low' refers to only the $190$ 
gene sets with $p$-values in the interval $(10^{-5},10^{-4})$
for Moran, or the $15$ gene sets 
with $p$-values in the interval $(10^{-5},10^{-3})$
for Scherzer.
}
\end{table}

While saddlepoint methods have a very desirable
relative error property they do have some numerical
issues.  Table~\ref{tab:timing} shows some timing
data. We also got infinite values for $95$ of the
gene sets  on the Zhang data. It might 
be a convergence issue or possibly a flaw in
how we implemented saddlepoints.  

\begin{table}
\begin{tabular}{lcccc}
Data Set & Saddle & $\hat p_1$ & $\hat p_2$ & $\hat p_3$ \\
\midrule
Zhang & 0.0631 & 0.0024 & 0.0031 & 0.0032\\
Moran & 0.0894 & 0.0029 & 0.0037 & 0.0038\\
Scherzer &  0.1394 & 0.0034 & 0.0045 & 0.0047\\
\bottomrule
\end{tabular}
\caption{\label{tab:timing}
Average, over 6180 gene sets of the  running time in seconds
for saddlepoint $p$-values 
and $\hat p_j$, $j=1,2,3$. Total time ranges from under $1/4$ minute
to just over $14$ minutes.
}
\end{table}

\section{Discussion}\label{sec:discussion}

We have constructed approximations to the permutation $p$-value 
using probability and geometry derived from discrepancy theory. 
A rigorous upper bound for $p$
could be attained using $L_\infty$ spherical cap discrepancies instead of 
the $L_2$ version, but computing such discrepancies is a major challenge. 
\cite{narc:sun:ward:wu:2010} give upper bounds for the $L_\infty$ spherical 
cap discrepancy, in terms of averages of a great many harmonic functions at 
the points $\bsx_i$. For our application we need bounds for 
spherical caps of a fixed volume (under distribution~\ref{mod:one}) 
and of fixed volume and constrained location (under distribution~\ref{mod:two}) 
and those go beyond what is in \cite{narc:sun:ward:wu:2010}. 

Many other approximation methods have been proposed for permutation tests. 
For instance, \cite{zhou2009efficient} fit approximations by moments 
in the Pearson family. 
\cite{lars:owen:2015} fit Gaussian and beta approximations 
to linear statistics and gamma approximations to quadratic 
statistics for gene set testing problems. 
\cite{knij:wess:rein:shmu:2009} fit generalized extreme value 
distributions to the tails of sampled permutation values. 

None of these approximations come with an all inclusive 
$p$-value that accounts for both numerical uncertainty
of the estimation and sampling uncertainty behind the
original data.  IID sampling of permutations does come with 
such a $p$-value if we add one to numerator 
and denominator as \cite{barn:1963} suggests. 
Then the Monte Carlo $p$-value estimate $\hat p$ is actually a conservative 
$p$-value in it's own right: $\Pr( \hat p \le u ) \le u$ for $0\le u\le 1$
under the null hypothesis. 
However, that method cannot attain very small $p$-values, and
so a gap remains.

We have employed reference distributions in an effort to
address this gap.  We select a set $\bby$ containing
$\bsy_0$ and find the first two moments of $p(\bsy,\hat\rho)$
for $\bsy\sim\dustd(\bby)$.
If the data $\bsy_0$ were actually sampled from our reference
distribution, then we could get an all inclusive conservative
$p$-value via the Chebychev inequality.

To illustrate the Chebychev inequality,
let $\mu=\e(p(\bsy,\hat\rho))$ and $\sigma^2 = \var(p(\bsy,\hat\rho))$
for the observed value $\hat\rho = \bsx_0^\tran\bsy_0$ and for random 
$\bsy\sim\dustd(\bby)$ for some reference set $\bby$.
Then  $\Pr( p\ge\mu+\lambda\sigma) \le 1/(1+\lambda^2)$ for any $\lambda>0$. 
Under this model, $p^*=\mu +\lambda\sigma+1/(1+\lambda^2)$ is a conservative $p$-value. 
Minimizing $p^*$ over $\lambda$ reduces to solving $2\lambda=\sigma(1+\lambda^2)^2$. 
For small $p$ we anticipate $\lambda\gg1$ and hence $\lambda'=(2/\sigma)^{1/3}$ will be 
almost as good as the optimal $\lambda$ we could find numerically. That choice leads to 
$p^*\le \mu +(2^{1/3}+2^{-2/3})\sigma^{2/3}$.   

For a numerical illustration, consider $\mu = 10^{-30}$
and $\sigma = 3\times 10^{-30}$, roughly describing the small $p$-value estimates 
from the case $m_0=m_1=70$. Then $p^*\le 4\times10^{-20}$, much larger than $\mu$
and yet still very small, likely small enough to be significant
after multiplicity adjustments.

The reference distributions describe a set in which $p(\bsy_0,\hat\rho)$ is known to lie. 
Reference distribution $1$ applies for Gaussian $Y_i$ and
reference distribution $2$ is for Gaussian $Y_i$ after conditioning
on $\bsy^\tran\bsx_0=\bsy_0^\tran\bsx_0$.
Of course, the data will not ordinarily be exactly Gaussian.  The numerical illustration
uses a Chebychev inequality at $\lambda'\doteq8.7\times 10^9$ standard deviations.

As mentioned above, an 
$L_\infty$ version of the Stolarsky inequality would eliminate the
need for Chebychev inequalities though it might also
be very conservative. We do not  know whether
$p(\bsy,\hat\rho)$ has a heavy tailed distribution under
reference distribution $2$.

Our permutation points fall into a lattice subset of $\real^d$
intersected with the unit sphere $\sd$.
Our problem of counting the number of such points in a
subset is one that is addressed under the term
`Geometry of numbers'.  According to a personal
communication from Neil Sloane, the standard approach
to such problems is via the volume ratio, which in
our setting is $\hat p_1$ which does not do well
for our problems.

To have reasonable power to obtain a $p$-value below $\epsilon$ by permutation sampling requires on the order of $1/\epsilon$ permutations each requiring
$O(n\log(n))$ computation to generate and $O(n)$ computation to evaluate
the inner product.
The cost to compute the standard errors
in our method is dominated by a cost proportional to $\minm^3$
though there is a very small cost proportional to $\minm^4$. In the
range where the first cost dominates, our proposal is advantageous
when $\minm^3 = o(n/\epsilon)$.  Supposing that $m_0$ and $m_1$
are comparable, our advantage holds when $\minm^2 = o(1/\epsilon)$.
If only the estimate and not the standard error is required, then our
$\hat p_2$ and $\hat p_3$ cost $O(\minm)$ once the $\hat\rho$
(cost $O(n)$) has been computed. Then the total cost is $O(n)$
compared to the much larger cost $O(n/\epsilon)$ for sampling.

\section*{Acknowledgements}
This work was supported by the US National Science Foundation under grants
DMS-1407397 and DMS-1521145.
We thank John Robinson and Neil Sloane for helpful comments.

\bibliographystyle{apalike}
\bibliography{stolarsky}

\section{Appendix}\label{sec:appendix}

Here we collect up some of the longer proofs.

\subsection{Proof of Theorem~\ref{thm:model1viastolarsky} (Limiting invariance)}
\label{sec:proof:thm:model1viastolarsky}
Here we show that taking limits as $\beps$ goes
to zero in the formula of \cite{brau:dick:2013}
proves Theorem~\ref{thm:model1viastolarsky}.
We use three lemmas, one for each term in
Theorem~\ref{thm:weighted-stolarsky}.
We use $\beps\to0$ as a shorthand for $\lim_{\epsilon_1\to0^+}\lim_{\epsilon_2\to0^+}$.

\begin{lemma}\label{lem:vepslhs}
Let $v_{\beps}$ be defined as in \eqref{eq:veps}. Then for $\hat\rho\in[-1,1)$,
\begin{align*}
&\lim_{\beps\to0} \int_{-1}^1v_{\beps}(t) \int_{\sd} \biggl| \sigma_d(C(\bsz;t)) -
      \frac{1}{N}\sum\limits_{k=0}^{N-1}\one_{C(\bsz;t)}(\bsx_k)\biggr|^2
    \rd \sigma_d(\bsz) \rd t \\
    &= \int_{\sd} |p(\bsz, \hat{\rho}) -\hat{p}(\hat{\rho})|^2\rd \sigma_d(\bsz).
\end{align*}
\end{lemma}
\begin{proof}
Substituting $v_{\beps}$ we get
\begin{align*}
&\quad
\int_{-1}^1\Bigl(\epsilon_2 + \frac{1}{\epsilon_1} \one (\hat{\rho}\le t \le\hat{\rho} + \epsilon_1)\Bigr)
\int_{\sd} (\hat p(t)-p(\bsz,t))^2 \rd \sigma_d(\bsz) \rd t\\
\to&\quad
\frac1{\epsilon_1}\int_{\hat\rho}^{\hat\rho+\epsilon_1}
\int_{\sd} (\hat p(t)-p(\bsz,t))^2 \rd \sigma_d(\bsz) \rd t,\quad\text{as $\epsilon_2\to0^+$}\\
\to&\quad
\int_{\sd} (\hat p(\hat\rho)-p(\bsz,\hat\rho))^2 \rd \sigma_d(\bsz),\quad\text{as $\epsilon_1\to0^+$}.\qedhere
\end{align*}
\end{proof}

\begin{lemma}\label{lem:kvsum}
Let $v_{\beps}$ be as in~\eqref{eq:veps} with $\hat\rho\in[-1,1)$,
and let $K_{v_{\beps}}$ be given by~\eqref{eq:defkv}.
Then for any $\bsx,\bsx'\in\sd$,
\begin{align*}
\lim_{\beps\to 0}K_{v_{\beps}}(\bsx, \bsx') = \sigma_d(C(\bsx;\hat{\rho}) \cap C(\bsx';\hat{\rho})).
\end{align*}
\end{lemma}
\begin{proof}
The argument is essentially the same as for Lemma \ref{lem:vepslhs}.
\end{proof}

\begin{lemma}\label{lem:kvint}
Let $v_{\beps}$ be as in~\eqref{eq:veps} with $\hat\rho\in[-1,1)$,
and let $K_{v_{\beps}}$ be given by~\eqref{eq:defkv}.
Then
\begin{align}\label{eq:kvint}
\lim_{\beps\to0}
\int_{\sd}  \int_{\sd}K_{v_{\beps}}(\bsx, \bsy)\rd \sigma_d(\bsx)\rd \sigma_d(\bsy) = \hat{p}_1(\hat{\rho})^2.
\end{align}
\end{lemma}
\begin{proof}
For any $\bsx, \bsy\in \sd$,
the kernel $K_{v_{\beps}}(\bsx, \bsy)$ is nonnegative and upper bounded by a constant.
Therefore we can take our limit operations inside the double integral over $\bsx$ and $\bsy$.
Now
$\lim_{\beps\to0}K_{v_{\beps}}(\bsx, \bsy)
 = \int_{\sd} \one_{C(\bsz;\hat{\rho})}(\bsx)\one_{{C(\bsz;\hat{\rho})}}(\bsy)\rd \sigma_d(\bsz)$.
Therefore the limit in~\eqref{eq:kvint} is
\begin{align*}
\int_{\sd}\int_{\sd}\int_{\sd}
 \one_{C(\bsz;\hat{\rho})}(\bsx)\one_{{C(\bsz;\hat{\rho})}}(\bsy)\rd \sigma_d(\bsz)\rd\sigma_d(\bsy)\rd\sigma_d(\bsx)
& = \hat{p}_1(\hat{\rho})^2
\end{align*}
after changing the order of the integrals.
\end{proof}

\par\noindent{\bf Theorem~\ref{thm:model1viastolarsky}}
{\sl
Let $\bsx_0,\bsx_1,\dots,\bsx_N\in\sd$ and $\hat\rho\in[-1,1]$. Then
\begin{align*}
 \int_{\sd} |p(\bsz, \hat{\rho}) -\hat{p}_1(\hat{\rho})|^2\rd \sigma_d(\bsz) = \frac{1}{N^2} \sum_{k=0}^{N-1} \sum_{l=0}^{N-1} \sigma_d(C(\bsx_k;\hat{\rho}) \cap C(\bsx_\ell;\hat{\rho})) - \hat{p}_1(\hat{\rho})^2.
\end{align*}
}
\begin{proof}
Theorem \ref{thm:weighted-stolarsky}  gives us an identity
and applying Lemmas~\ref{lem:vepslhs},  \ref{lem:kvsum}  and \ref{lem:kvint}
to both sides of it establishes~\eqref{eq:phat1viastolarsky} for $\rho\in[-1,1)$.
For $\hat\rho=1$ we get the answer by replacing $v_{\beps}$
by $\epsilon_2 + (1/\epsilon_1)\one_{1-\epsilon_1\le t\le 1}$ in the lemmas.
Replacing $\bsz$ by $\bsy$ and $\hat\rho$ by $t$ above gives the version in the main body of the article.
\end{proof}

\subsection{Proof of Lemma~\ref{lem:P2} (Double inclusion for Model~\ref{mod:two})}\label{sec:proof:lem:P2}

\begin{proof}
We split the proof into four cases and prove them individually.
Recall that $P_2(u_1,u_2,u_3,\tilde\rho,\hat\rho)
= \int_{\sdmo}
\one(\langle\bsy,\bsx_1\rangle\ge\hat\rho)
\one(\langle\bsy,\bsx_2\rangle\ge\hat\rho)\rd\sigma_{d-1}(\bsy^*)$
where $\bsy=\tilde\rho\bsx_c +\sqrt{1-\tilde\rho^2}\bsy^*$.

 \textbf{Case 1}. $\bsx_1 = \bsx_2 = \bsx_c$, i.e., $r_1 = r_2 = r_3 = 0$.
 \begin{equation*}
P_2(1, 1, 1,\tilde{\rho},\hat{\rho}) = \int_{\sdmo} \one(\left<\bsy, \bsx_c \right>
\ge \hat{\rho})\one(\left<\bsy, \bsx_c \right>
\ge \hat{\rho})\rd \sigma_{d-1}(\bsy^{*}) = \one(\tilde{\rho}
\ge \hat{\rho}).
 \end{equation*}

\textbf{Case 2}. $\bsx_1 = \bsx_c\ne \bsx_2$, i.e., $r_1 = 0, r_2 > 0, r_3 > 0$.
\begin{equation*}
\begin{aligned}
P_2(1, u_2, u_2,\tilde{\rho},\hat{\rho}) &= \int_{\sdmo} \one(\left<\bsy, \bsx_c \right>
\ge \hat{\rho})\one(\left<\bsy, \bsx_2 \right>
\ge \hat{\rho})\rd \sigma_{d-1}(\bsy^{*}) \\
&= \one(\tilde{\rho} \ge \hat{\rho}) \int_{\sdmo} \one(\left<\bsy, \bsx_2 \right>
\ge \hat{\rho})\rd \sigma_{d-1}(\bsy^{*}) \\
&= \one(\tilde{\rho}\ge \hat{\rho}) P_1(u_2, \tilde{\rho}, \hat{\rho})
\end{aligned}
\end{equation*}
where the last step uses Lemma \ref{lem:P1}.

\textbf{Case 3}. $\bsx_1 = \bsx_2 \neq \bsx_c$, i.e., $r_1 = r_2 > 0=r_3$.
\begin{equation*}
\begin{aligned}
P_2(u_2, u_2, 1,\tilde{\rho},\hat{\rho}) &= \int_{\sdmo} \one(\left<\bsy, \bsx_1 \right>
\ge \hat{\rho})\one(\left<\bsy, \bsx_2 \right>
\ge \hat{\rho})\rd \sigma_{d-1}(\bsy^{*}) \\
&= \int_{\sdmo} \one(\left<\bsy, \bsx_2 \right>
\ge \hat{\rho})\rd \sigma_{d-1}(\bsy^{*}) \\
&= P_1(u_2, \tilde{\rho}, \hat{\rho}).
\end{aligned}
\end{equation*}

\textbf{Case 4}.  $\bsx_1 \neq \bsx_2 \neq \bsx_c\ne \bsx_1$,  i.e., $r_1, r_2, r_3 > 0$.
We split this case into subcases. First we assume $u_2 = -1$, so
\begin{equation*}
\begin{aligned}
P_2(u_1, u_2, u_3,\tilde{\rho},\hat{\rho}) &= \int_{\sdmo} \one(\left<\bsy,\bsx_1 \right>
\ge \hat{\rho})\one(\left<\bsy, -\bsx_c \right>
\ge \hat{\rho})\rd \sigma_{d-1}(\bsy^{*}) \\
&=\one(-\tilde{\rho}
\ge \hat{\rho}) \int_{\sdmo} \one(\left<\bsy, \bsx_1 \right>
\ge \hat{\rho})\rd \sigma_{d-1}(\bsy^{*}) \\
& = \one(-\tilde{\rho}
\ge \hat{\rho}) P_1(u_1, \tilde{\rho}, \hat{\rho}).
\end{aligned}
\end{equation*}
Similarly if $u_1 = -1$, then
\[
P_2(u_1, u_2, u_3,\tilde{\rho},\hat{\rho}) = \one(-\tilde{\rho}
\ge \hat{\rho}) P_1(u_2, \tilde{\rho}, \hat{\rho}).
\]

Finally we assume $u_1 > -1$ and $u_2 > -1$, so now  $|u_1| < 1$ and $|u_2| < 1$.
Recall the projections $\bsx_j=u_j\bsc_c +\sqrt{1-u_j^2}\bsx_j^*$ for $j=1,2$
and introduce further projections of $\bsy^*$ and $\bsx_2^*$ onto $\bsx_1^*$:
$\bsy^{*} = t\bsx_1^{*} + \sqrt{1-t^2}\bsy^{**}$ and 
$\bsx_2^{*}  = u_3^{*}\bsx_1^{*} + \sqrt{1-u_3^{*}}\bsx_2^{**}$.
The residuals $\bsy^{**}$ and $\bsx_2^{**}$ belong to a subset of $\sd$ that
is isomorphic to $\sdmt$.
Now we have
\begin{align*}
&P_2(u_1, u_2, u_3,\tilde{\rho},\hat{\rho})  \\
&=\int_{\sdmo} \one(\left<\bsy, \bsx_1 \right>
\ge \hat{\rho})\one(\left<\bsy, \bsx_2 \right>
\ge \hat{\rho})\rd \sigma_{d-1}(\bsy^{*})\\
&= \int_{-1}^1 \frac{\omega_{d-2}}{\omega_{d-1}}(1-t^2)^{\frac{d-1}{2} - 1}\int_{\sdmt}
\one\Bigl(\tilde{\rho}u_1 +  \sqrt{1-\tilde{\rho}^2}\sqrt{1-u_1^2}t
\ge \hat{\rho}\Bigr) \\
&\qquad \times \one\Bigl(\tilde{\rho}u_2 +
\sqrt{1-\tilde{\rho}^2}\sqrt{1-u_2^2}(tu_3^{*} +
\sqrt{1-t^2}\sqrt{1-u_3^{*2}}\left<\bsy^{**}, \bsx_2^{**} \right>)
\ge \hat{\rho}\Bigr) \\
&\qquad \times \rd \sigma_{d-1}(\bsy^{**}) \rd t\\
&=
\int_{-1}^1 \frac{\omega_{d-2}}{\omega_{d-1}}(1-t^2)^{\frac{d-1}{2} - 1}
\one\Bigl(  t \ge \frac{\hat{\rho} -
  \tilde{\rho}u_1}{\sqrt{1-\tilde{\rho}^2}\sqrt{1-u_1^2}}\Bigr) \\
&\qquad \times \int_{\sdmt}\one\Bigl(\tilde{\rho}u_2 +
\sqrt{1-\tilde{\rho}^2}\sqrt{1-u_2^2}(tu_3^{*} +
\sqrt{1-t^2}\sqrt{1-u_3^{*2}}\left<\bsy^{**}, \bsx_2^{**} \right>)
\ge \hat{\rho}\Bigr)\\
&\qquad\times
\rd \sigma_{d-1}(\bsy^{**}) \rd t\\
& =
\begin{cases}
\one(\tilde{\rho}u_1 \ge \hat{\rho})\one(\tilde{\rho}u_2 \ge
\hat{\rho}), &\tilde{\rho} = \pm 1\\
\int_{-1}^1 \frac{\omega_{d-2}}{\omega_{d-1}}(1-t^2)^{\frac{d-1}{2} - 1} \one(t \ge \rho_1)  \one(tu_3^{*} \ge \rho_2) \rd t,
& \tilde{\rho} \neq \pm 1, u_3^{*} = \pm 1\\
\int_{-1}^1 \frac{\omega_{d-2}}{\omega_{d-1}}(1-t^2)^{\frac{d-1}{2} - 1} \one(t \ge \rho_1)
\sigma_{d-2}(C(\bsx_2^{**}, \frac{\rho_2 -
  tu_3^{*}}{\sqrt{1-t^2}\sqrt{1-u_3^{*2}}})) \rd t, & \tilde{\rho}
\neq \pm 1, |u_3^{*}| < 1
\end{cases}
\end{align*}
where $u_3^*, \rho_1, \rho_2$ are defined in \eqref{eq:rho12}. Hence, the result follows.
\end{proof}

\subsection{Proof of Theorem~\ref{thm:phat2} (Second moment under 
reference distribution~\ref{mod:two})}
\label{sec:proof:thm:phat2}

\begin{proof}
Without loss of generality we
relabel the values $\bsx_k$ so that $c=0$.  Any other choice for $c$
is reflected in the number $\wt\rho$.
The second moment is
\begin{align}\label{eq:secmomentsum}
\e(p(\bsy, \hat{\rho})^2 )
&=\frac{1}{N^2}
\sum\limits_{k= 0}^{N-1}
\sum\limits_{l=0}^{N-1}P_2(u_k, u_\ell, u_{k,\ell}, \tilde{\rho},
\hat{\rho})
\end{align}
where $u_{k,\ell}$ is obtained via~\eqref{eq:ur} from the swap distance $r_{k,\ell}$ 
between points $\bsx_k$ and $\bsx_\ell$.
We will partition the sum in~\eqref{eq:secmomentsum} into the same four cases as in the proof of Lemma \ref{lem:P2}.

\textbf{Case 1}, $\bsx_k = \bsx_\ell = \bsx_c$, i.e., $r_k = r_\ell = r_{k,\ell}= 0$.
There is only one pair of $(\bsx_k,\bsx_\ell)$ for this condition.
Hence, we get only one term corresponding to $P_2(1,1,1,\tilde{\rho},\hat{\rho}) = \one(\tilde{\rho} \ge \hat{\rho})$.

\textbf{Case 2}, $\bsx_k = \bsx_c\ne \bsx_\ell$, i.e., $r_k = 0, r_\ell = r_{k,\ell} > 0$.
Consider all pairs of $(\bsx_k,\bsx_\ell)$ that satisfy this condition and let $K_2$ denote their
total contribution to~\eqref{eq:secmomentsum}. Then
\begin{equation*} 
  \begin{aligned}
K_{2} &= 2\sum\limits_{l = 1}^{N-1} \int_{\sdmo}\one(\left<\bsy, \bsx_c \right>
\ge \hat{\rho})\one(\left<\bsy, \bsx_\ell \right>
\ge \hat{\rho}) \rd \sigma_{d-1}(\bsy^{*})
\\ &
= 2\sum\limits_{r= 1}^{\minm} {m_0 \choose r}{m_1  \choose r}P_2(1, u(r), u(r),\tilde{\rho}, \hat{\rho}).
  \end{aligned}
\end{equation*}

\textbf{Case 3}, $\bsx_k = \bsx_\ell \neq \bsx_c$, i.e., $r_k = r_\ell > 0 = r_{k,\ell}$.
The contribution from terms of this form is
\begin{equation*} 
  \begin{aligned}
K_{3}
&= \sum\limits_{k= 1}^{N-1} \int_{\sdmo}\one(\left<\bsy, \bsx_k \right>
\ge \hat{\rho}) \rd \sigma_{d-1}(\bsy^{*})
= \sum\limits_{r = 1}^{\minm} {m_0 \choose r}{m_1  \choose r}
P_1(u(r),\tilde{\rho}, \hat{\rho}).
  \end{aligned}
\end{equation*}

\textbf{Case 4,}  $\bsx_k \neq \bsx_\ell \neq \bsx_c$,  i.e., $r_k, r_\ell, r_{k,\ell}> 0$.
The contribution of these cases to the sum is
\begin{equation*} 
  \begin{aligned}
K_4 = &\sum_{k=1}^{N-1}\sum_{\ell=1}^{N-1}\one(\ell\ne k)
 \int_{\sdmo}\one(\left<\bsy, \bsx_i \right>
\ge \hat{\rho})\one(\left<\bsy, \bsx_j \right>
\ge \hat{\rho}) \rd \sigma_{d-1}(\bsy^{*}) \\
&=\sum_{r_k\in R}\sum_{r_\ell\in R}\sum_{r_{k,\ell}\in R_3(\bsr)}c(r_k,r_\ell,r_{k,\ell}) P_2(u_1,u_2,u_3,\tilde\rho,\hat\rho).
\end{aligned}
\end{equation*}
Then the second moment is $(\one(\tilde\rho\ge\hat\rho)+K_2+K_3+K_4)/N^2$.
\end{proof}

\subsection{Proof of Theorem~\ref{thm:location-weighted-stolarsky}
(Location weighted invariance)
}
\label{sec:proof:thm:location-weighted-stolarsky}
\begin{proof}
We follow the  technique in \cite{brau:dick:2013}. We begin by showing that $\Kv$ as defined in \eqref{eq:generalized-kernel} is a reproducing kernel.
First,  $\Kv$ is symmetric: $\Kv(\bsx, \bsy) = \Kv(\bsy, \bsx)$.
Next, choose $a_0, \dots, a_{N-1} \in \real$ and $\bsx_0, \dots, \bsx_{N-1} \in \sd$.
Then $\sum\limits_{k, \ell = 0}^{N-1} a_ka_\ell\Kv(\bsx_k, \bsx_\ell)$ equals
\begin{align*}
&\int_{-1}^1\int_{\sd} \sum\limits_{k, \ell = 0}^{N-1}
a_ka_\ell v(t)h(\left<\bsz, \bsx' \right>
)\one_{C(\bsz;t)}(\bsx_k)\one_{C(\bsz;t)}(\bsx_\ell)\rd \sigma_d(\bsz)\rd t\\
=&\int_{-1}^1
\int_{\sd}v(t) h(\left<\bsz,
\bsx' \right>) \biggl|\sum\limits_{k=0}^{N-1}a_k\one_{C(\bsz,t)}(\bsx_k)\biggr|^2 \rd \sigma_d(\bsz)\rd t
\end{align*}
which is nonnegative.
Thus $\Kv$ is symmetric and positive definite, and so by \cite{aron:1950},
$\Kv$ is a reproducing kernel.

\cite{aron:1950} also shows that a reproducing
kernel uniquely defines a Hilbert space of functions with a specific
inner product. Let $\mathcal{H}_{v, h, \bsx'} = \mathcal{H}(K_{v, h, \bsx'}, \sd)$ denote
the corresponding reproducing kernel Hilbert space of functions $f:
\sd \to \real$ with reproducing kernel $\Kv$.

We now consider functions $f_1, f_2: \sd \to\real$ which
admit the representation
\begin{equation}  \label{eq:hilbert-function}
f_i(\bsx) =
\int_{-1}^1\int_{\sd}g_i(\bsz;t)\one_{C(\bsz;t)}(\bsx)\rd \sigma_d(\bsz)\rd t,
\quad i =
1, 2
\end{equation}
for functions $g_i\in L_2(\sd\times[-1,1])$.
For any fixed $\bsy \in \sd$ the function $\Kv(\cdot, \bsy)$ has
representation~\eqref{eq:hilbert-function} via
$g(\bsz;t) = v(t)h(\left<\bsz, \bsx'\right>)\one_{C(\bsz, t)}(\bsy)$.

For functions with representation~\eqref{eq:hilbert-function}, we define an inner product by
\begin{equation}\label{eq:inner-product}
\left<f_1,f_2 \right>_{\Kv} =
\int_{-1}^1\frac{1}{v(t)}\int_{\sd}\frac{1}{h(\left<\bsz, \bsx' \right>
)}g_1(\bsz, t)g_2(\bsz,t)\rd \sigma_d(\bsz)\rd t.
\end{equation}
For $\bsy \in \sd$ and $f_1\in\mathcal{H}_{v,h,\bsx'}$,
\begin{align*}
\left<f_1, \Kv(\cdot, \bsy) \right>_{\Kv} &=
\int_{-1}^1 \frac1{v(t)}\int_{\sd} \frac{{g_1(\bsz; t)v(t)}}{h(\left<\bsz, \bsx' \right>)}
h(\left<\bsz, \bsx' \right>
)\one_{C(\bsz;t)}(\bsy)\rd \sigma_d(\bsz)\rd t \\
&= \int_{-1}^1\int_{\sd}g_1(\bsz, t)\one_{C(\bsz,t)}(\bsy)\rd \sigma_d(\bsz)\rd t\\
& = f_1(\bsy),
\end{align*}
showing that the inner product~\eqref{eq:inner-product} has the reproducing property.
By \cite{aron:1950}, the inner product in $\Hv$ is unique.
Functions $f_i$ satisfying \eqref{eq:hilbert-function}
with $\left<f_i, f_i\right>_{\Kv} < \infty$ are in $\Hv$, and
\eqref{eq:inner-product} is the unique inner product of $\Hv$.

We prove the theorem by equating two different forms of $\|\mathcal{R}(\Hv; \cdot\,)\|_{\Kv}$ where
\[
\mathcal{R}(\Hv; \cdot\,) = \int_{\sd}\Kv(\cdot\,, \bsy)\rd \sigma_d(\bsy) -
\frac{1}{N}\sum\limits_{k = 0}^{N-1}\Kv(\cdot\,, \bsx_k). 
\]
Although $\mathcal{R}(\Hv;\cdot\,)$ depends on our specific points
$\bsx_i$ we omit that from the notation.
The reproducing property of $\Kv$ yields
\begin{align*}
\left<\Kv(\cdot\,, \bsx_k), \Kv(\cdot\,, \bsx_\ell) \right>_{\Kv} &= \Kv(\bsx_k, \bsx_\ell)
\end{align*}
from which it follows that
\begin{equation}  \label{eq:kernel-set}
\begin{aligned}
&\left< \int_{\sd}\Kv(\cdot\,, \bsy)\rd \sigma_d(\bsy),
            \int_{\sd}\Kv(\cdot\,, \bsy')\rd \sigma_d(\bsy')\right>_{\Kv}\\
&=\int_{\sd}\int_{\sd}\Kv(\bsy, \bsy')\rd \sigma_d(\bsy)\rd \sigma_d(\bsy').
\end{aligned}
\end{equation}
Using  \eqref{eq:kernel-set} and the linearity of the inner product, we have
\begin{equation}  \label{eq:representer1}
  \begin{aligned}
&\left<\mathcal{R}(\Hv; \cdot\,), \mathcal{R}(\Hv; \cdot\,)\right>_{\Kv}\\
&= \int_{\sd}\int_{\sd}\Kv(\bsy,\bsy')\rd \sigma_d(\bsy)\rd \sigma_d(\bsy')
- \frac{2}{N}\sum\limits_{k=0}^{N-1}\int_{\sd}\Kv(\bsy, \bsx_k)\rd \sigma_d(\bsy) \\
&+ \frac{1}{N^2}\sum\limits_{k,\ell=0}^{N-1}\Kv(\bsx_k, \bsx_\ell).
\end{aligned}
\end{equation}

For our second form of $\|\mathcal{R}(\Hv; \cdot\,)\|_{\Kv}$, we write
\begin{equation*}\begin{aligned}
&  \mathcal{R}(\Hv; \cdot\,) \\
=&
  \int_{\sd}\Kv(\cdot\,,\bsy) \rd \sigma_d(\bsy) - \frac{1}{N}\sum\limits_{k = 0}^{N-1}\Kv(\cdot\,, \bsx_k)\\
=& \int_{-1}^1v(t)\int_{\sd}\one_{C(\bsz;t)}(\bsx) \hzx
\biggl[\int_{\sd}\one_{C(\bsz,t)}(\bsy) \rd \sigma_d(\bsy)\rd t -
  \frac{1}{N}\sum\limits_{k = 0}^{N-1}\one_{C(\bsz, t)}(\bsx_k) \biggr]
\rd \sigma_d(\bsz)\rd t\\
= & \int_{-1}^1v(t)\int_{\sd}\one_{C(\bsz,t)}(x)\hzx
\biggl[\sigma_d(C(\bsz, t)) - \frac{1}{N}\sum\limits_{k = 0}^{N-1}\one_{C(\bsz, t)}(\bsx_k) \biggr]\rd \sigma_d(\bsz)\rd t.
\end{aligned}
\end{equation*}
Hence using the definition of the inner product $\left<\cdot, \cdot \right>_{\Kv}$, we have
\begin{equation}  \label{eq:representer2}
\begin{aligned}
&\left<\mathcal{R}(\Hv; \bsx), \mathcal{R}(\Hv; \bsx)\right>_{\Kv}\\
=& \int_{-1}^1v(t)\int_{\sd}\hzx \biggl|\sigma_d(C(\bsx,t)) -
    \frac{1}{N}\sum\limits_{k = 0}^{N-1}\one_{C(\bsx; t)}(\bsx_k)\biggr|^{2} \rd \sigma_d(\bsx)\rd t.
\end{aligned}
\end{equation}
Combining equations \eqref{eq:representer1} and \eqref{eq:representer2},
we have the generalized location-weighted version of the Stolarsky invariance principle.
\end{proof}

\subsection{Proof of Theorem~\ref{thm:invariancemodeltwosecondmoment}
(Spatially weighed invariance) }
\label{sec:proof:thm:invariancemodeltwosecondmoment}
As in Section~\ref{sec:proof:thm:model1viastolarsky},
$\lim_{\beps\to0}$ means
$\lim_{\epsilon_1\to 0^+}\lim_{\epsilon_2\to0^+}$
and similarly $\lim_{\bseta\to0}$ denotes
$\lim_{\eta_1\to 0^+}\lim_{\eta_2\to0^+}$.
We prove  a series of lemmas first.

\begin{lemma}\label{lem:vhepslhs}
For $v_{\beps}(\cdot)$ and $h_{\bseta}(\cdot)$ defined by equations~\eqref{eq:veps}
and~\eqref{eq:heta},
\begin{align*}
&\lim_{\bseta\to0}\lim_{\beps\to0}
  \int_{-1}^1v_{\beps}(t)\int_{\sd} h_{\bseta}(\langle \bsz, \bsx_c\rangle ) \biggl|\sigma_d(C(\bsz,t)) -
    \frac{1}{N}\sum\limits_{k = 0}^{N-1}\one_{C(\bsz; t)}(\bsx_k)\biggr|^{2}
  \rd \sigma_d(\bsz)\rd t \\
&= \int_{\sdmo}|p(\tilde{\rho}\bsx_c + \sqrt{1 - \tilde{\rho}^2}\bsy^{*}, \hat{\rho}) -
\hat{p}_1(\hat{\rho})|^{2}\rd \sigma_{d -1}(\bsy^{*}),
\end{align*}
where $\hat p_1(\hat\rho) = \sigma_d(C(\bsy;\hat\rho))$.
\end{lemma}

\begin{proof}
This proof is similar to the others.  First we take the limit $\beps\to0$ yielding
\begin{align*}
\lim_{\bseta\to 0}
\int_{\sd}h_{\bseta}(\langle \bsz, \bsx_c\rangle ) \biggl|\sigma_d(C(\bsz,\hat{\rho})) -
    \frac{1}{N}\sum\limits_{k = 1}^{N}\one_{C(\bsz; \hat{\rho})}(\bsx_k)\biggr|^{2}
  \rd \sigma_d(\bsz).
\end{align*}
Making the projection $\bsz =  s\bsx_c  + \sqrt{1-s^2}\bsz^{*}$ gives
\begin{align*}
 & \lim_{\bseta\to 0}
\int_{-1}^{1}\int_{\sdmo}\frac{\omega_{d-1}}{\omega_{d}}(1-s^2)^{d/2
  - 1} h_{\bseta}(s)\,\times\\
& \qquad\biggl|\sigma_d(C(s\bsx_c + \sqrt{1-s^2}\bsz^{*},\hat{\rho})) -
    \frac{1}{N}\sum\limits_{k = 1}^{N}\one_{C(s\bsx_c + \sqrt{1-s^2}\bsz^{*}; \hat{\rho})}(\bsx_k)\biggr|^{2}\rd
\sigma_{d-1}(\boldsymbol{z}^{*}) \rd s \\
& = \int_{\sdmo}|p(\hat{\rho}\bsx_c + \sqrt{1 - \hat{\rho}^2}\bsy^{*}, \hat{\rho}) -
\hat{p}_1(\hat{\rho})|^{2}\rd \sigma_{d -1}(\bsy^{*}). \qedhere
\end{align*}
\end{proof}

\begin{lemma}\label{lem:double-sum-limit}
For $v_{\beps}(\cdot)$ and $h_{\bseta}(\cdot)$ defined by equations~\eqref{eq:veps} and~\eqref{eq:heta} ,
\begin{align*}
& \lim_{\bseta \to  0}  \lim_{\beps \to  0}
\frac{1}{N^2}\sum\limits_{k,\ell=0}^{N-1}\Kvhepsc(\bsx_k, \bsx_\ell) \\
&= \frac{1}{N^2}\sum\limits_{k,\ell=0}^{N-1}\int_{\sdmo}\one(\left<\bsy, \bsx_k \right>
\ge \hat{\rho})\one(\left<\bsy, \bsx_\ell \right>
\ge \hat{\rho}) \rd \sigma_{d-1}(\bsy^{*}).
\end{align*}
\end{lemma}

\begin{proof}
First, $\lim_{\bseta \to  0} \lim_{\beps \to  0}
N^{-2}\sum\limits_{k,\ell=0}^{N-1}\Kvhepsc(\bsx_k, \bsx_\ell) $
equals
\begin{align*}
\frac{1}{N^2}\sum\limits_{k,\ell=0}^{N-1} \lim_{\bseta \to 0} \int_{\sd}\heps(\left<\bsz, \bsx_c \right>
)\one_{C(\bsz;\hat{\rho})}(\bsx_k)\one_{C(\bsz;\hat{\rho})}(\bsx_\ell)\rd \sigma_d(\bsz).
\end{align*}
Projecting $\bsz$ onto $\bsx_c$ yields
$\bsz = s \bsx_c + \sqrt{1 - s^2} \bsy^*$ and then we have
\begin{align*}
&  \frac{1}{N^2}\sum\limits_{k,\ell=0}^{N-1} \lim_{\bseta \to 0} \int_{-1}^1 \frac{\omega_{d-1}}{\omega_d} (1-s^2)^{d/2 - 1} \heps(s) \int_{\sdmo} \one_{C(\bsz;\hat{\rho})}(\bsx_k)\one_{C(\bsz;\hat{\rho})}(\bsx_\ell)\rd \sigma_{d-1}(\bsy^*) \\
& = \frac{1}{N^2}\sum\limits_{k,\ell=0}^{N-1}\int_{\sdmo}\one(\left<\bsy, \bsx_k \right>
\ge \hat{\rho})\one(\left<\bsy, \bsx_\ell \right>
\ge  \hat{\rho}) \rd \sigma_{d-1}(\bsy^{*}).\qedhere
\end{align*}
\end{proof}

\begin{lemma}\label{lem:double_int_limit}
For $v_{\beps}(\cdot)$ and $h_{\bseta}(\cdot)$ defined by equations~\eqref{eq:veps} and~\eqref{eq:heta},
\[
 \lim_{\bseta \to  0} \lim_{\beps \to  0}
\int_{\sd}\int_{\sd}\Kvhepsc (\bsx, \bsy)
\rd\sigma_d(\bsx) \rd \sigma_d(\bsy)
= \hat{p}_1(\hat{\rho})^{2}
\]
\end{lemma}

\begin{proof}
Because $\Kvhepsc$ is nonnegative and uniformly bounded
we may take the limit over $\beps$ inside the integrals. Now
$$
\lim_{\beps \to  0}\Kvhepsc (\bsx, \bsy)
=\int_{\sd}h_{\bseta}(\langle \bsz, \bsx_c\rangle )
\one_{C(\bsz;\hat{\rho})}(\bsx)\one_{C(\bsz;\hat{\rho})}(\bsy)\rd \sigma_d(\bsz),
$$
and the limit becomes
\begin{align*}
\lim_{\bseta \to 0}\int_{\sd}\int_{\sd}\int_{\sd}
h_{\bseta}(\langle \bsz, \bsx_c\rangle )
\one_{C(\bsz;\hat{\rho})}(\bsx)\one_{C(\bsz;\hat{\rho})}(\bsy)
\rd \sigma_d(\bsz)\rd\sigma_d(\bsx) \rd \sigma_d(\bsy).
\end{align*}
Integrating over $\bsz$ last we get
$\lim_{\bseta\to0}\int_{\sd}h_{\bseta}(\langle \bsz,\bsx_c\rangle) \hat p_1^2(\hat\rho)\rd\bsz
=\hat p_1^2(\hat\rho)$.
\end{proof}

\begin{lemma}\label{lem:sum_int_limit}
Under reference distribution \ref{mod:two}
\begin{align*}
\lim_{\bseta \to 0}\lim_{\beps \to 0}
\frac{2}{N}\sum\limits_{k=
  0}^{N-1}\int_{\sd}\Kvhepsc(\bsx, \bsx_k)\rd \sigma_d(\bsx) = 2\hat{p}_1(\hat{\rho}) \e(p(\bsy, \hat{\rho})).
  \end{align*}
\end{lemma}
\begin{proof}
The argument here is similar to the one used for Lemma~\ref{lem:double_int_limit}.
Take the limit over $\beps$ inside the integral and change the order of integration to yield
\begin{align*}
\lim_{\bseta \to 0}
2\hat{p}_1(\hat{\rho})
\int_{\sd}\frac1N\sum\limits_{k = 0}^{N -1}\heps(\left<\bsz,\bsx_c \right>)\one_{C(\bsz,
\hat{\rho})}(\bsx_k)\rd \sigma_d(\bsz).
  \end{align*}
Substituting the projection $\bsz = t\bsx_c +\sqrt{1-t^2}\bsz^*$ produces
\begin{align*}
&\lim_{\bseta \to 0} 2\hat{p}_1(\hat{\rho})\int_{-1}^1
\frac{\omega_{d-1}}{\omega_{d}}(1-t^2)^{d/2-1}\heps(t)\int_{\sdmo}
\frac{1}{N}\sum\limits_{k  = 0}^{N -1}
\one_{C(t\bsx_c + \sqrt{1-t^2}\bsz^{*},\hat{\rho})}(\bsx_k)\rd \sigma_{d-1}(\bsz^{*}) \rd t\\
&=2\hat{p}_1(\hat{\rho})
\int_{\sdmo}\frac1N\sum\limits_{k = 0}^{N -1}
\one_{C(\tilde\rho \bsx_c + \sqrt{1-\tilde\rho^2}\bsz^{*},\hat{\rho})}(\bsx_k)\rd \sigma_{d-1}(\bsz^{*}) \\
& = 2\hat p_1(\hat\rho)\e( p(\bsy,\hat\rho))
  \end{align*}
for $\bsy$ under Model~\ref{mod:two}.
\end{proof}

\par\noindent{\bf Proof of Theorem~\ref{thm:invariancemodeltwosecondmoment}}
\begin{proof}
The proof follows from using Lemmas \ref{lem:vhepslhs} to \ref{lem:sum_int_limit} and Theorem \ref{thm:location-weighted-stolarsky}.
\end{proof}

\end{document}